\documentclass[11pt]{article}
\usepackage{amsfonts, amsthm, epsfig, amsmath, amssymb, color,verbatim}
\usepackage{textcomp}
\usepackage{paralist}

\usepackage[english]{babel}

\usepackage{color}

\newcommand{\rone}{\mathbb{R}}
\renewcommand{\Re}{\rone}

\newcommand{\E}{\mathbf{E}}
\newcommand{\wt}{\widetilde}
\renewcommand{\P}{\mathbf{P}}

\newcommand{\mbR}{\mathbb{R}}

\let\oldmarginpar\marginpar
\renewcommand{\marginpar}[1]{\oldmarginpar{\scriptsize\texttt{\color{red}{#1}}}}

\newcommand{\ve}{\varepsilon}

\newcommand{\eps}{\varepsilon}

\newcommand{\e}{\varepsilon}

\newcommand{\sgn}{\mathrm{sgn}}

\newcommand{\1}{1\!\!\!\;{\rm I}}

\theoremstyle{plain}
\newtheorem{thm}{Theorem}[section]
\newtheorem{prop}{Proposition}[section]

\newtheorem{lem}{Lemma}[section]
\theoremstyle{definition}

\theoremstyle{remark}
\newtheorem{rem}{Remark}[section]

\usepackage{hyperref}

\DeclareMathSymbol{\ophi}{\mathalpha}{letters}{"1E}

\renewcommand{\phi}{\varphi}

\newcommand{\be}{\begin{equation}}
\newcommand{\bel}{\begin{equation}\label}
\newcommand{\ee}{\end{equation}}
\newcommand{\ben}{\begin{equation*}}
\newcommand{\een}{\end{equation*}}

\newcommand{\ba}{\begin{aligned}}
\newcommand{\ea}{\end{aligned}}

\newcommand{\cF}{\mathcal{F}}

\newcommand{\bR}{\mathbb{R}}

\newcommand{\vf}{\varphi}

\renewcommand{\d}{\mathrm{d}}

\newfont{\cyrfnt}{wncyr10}
\def\J3{\cyrfnt{\rm \u{\cyrfnt I}}}
\def\j3{\cyrfnt{\rm \u{\cyrfnt i}}}

\usepackage[]{color}
\definecolor{DarkGreen}{rgb}{0.1,0.7,0.3}   

\definecolor{DarkGreen}{rgb}{0.1,0.7,0.3}   

\numberwithin{equation}{section}

\allowdisplaybreaks[4]

\begin{document}

\title{On regularization by a small
noise of multidimensional ODEs with non-Lipschitz coefficients }
\author{Alexei Kulik\footnote{Wroclaw University of Science and Technology, Faculty of Pure and Applied Mathematics,
Wybrze\'ze Wyspia\'nskiego Str. 27, 50-370 Wroclaw, Poland; kulik.alex.m@gmail.com}\ \ 
 and Andrey Pilipenko\footnote{Institute of Mathematics, National Academy of Sciences of Ukraine, Tereshchenkivska Str.\ 3, 01601, Kiev, 
Ukraine} \footnote{National Technical University of Ukraine 
``Igor Sikorsky Kyiv Polytechnic Institute'',
ave.\ Pobedy 37, Kiev 03056, Ukraine; pilipenko.ay@gmail.com}}

\maketitle
\begin{abstract}
In this paper we solve a 
 selection problem for multidimensional SDE 
 $d X^\e(t)=a(X^\e(t))\, d t+\eps \sigma(X^\e(t))\, d W(t)$, 
where the drift and diffusion are locally Lipschitz continuous outside of a fixed hyperplane $H$.
 It is assumed that  $X^\e(0)=x^0\in H$, the drift $a(x)$ has 
a Hoelder asymptotics as $x$ approaches $H$, and  the limit ODE  $d X(t)=a(X(t))\, d t$
does not have a unique solution.

We show that if the drift pushes the solution away of  $H$, then the limit process  with certain probabilities selects some extremal solutions to the 
limit ODE.  If the drift attracts the solution to  $H$, then the limit process
satisfies an ODE with some averaged coefficients. To prove the last result we formulate
  an averaging principle, which is quite general and
   new.
\end{abstract}
\section{Introduction}

Consider an ODE
\bel{eq:ODE}
\ba
\frac{d u(t)}{dt}= a(u(t));\\
u(0)=0,
\ea
\ee
where  $a$ is a continuous  function of linear growth that satisfies a local  Lipschitz  condition
 everywhere except of  the point $u=0$.  Then uniqueness of the solution
to \eqref{eq:ODE} may fail;   e.g. for  $a(u)=\sqrt{|u|}\sgn (u)$ the ODE \eqref{eq:ODE} has multiple solutions $\pm t^2/4, t\geq 0.$

Consider a perturbation of  \eqref{eq:ODE}   by a small noise:
\bel{eq:SDE_small}
{d u_\ve(t)}= a(u_\ve(t)) dt+ \ve dW(t),
\ee
$$
u_\ve(0)=0,
$$
where $W$ is a Wiener process. 
Equation \eqref{eq:SDE_small} has a unique
 strong solution due to  the Zvonkin-Veretennikov theorem \cite{Veretennikov}.
It easy to  see that a family  of distributions of $\{u_\ve\}$ is weakly relatively compact because
 $a$ has a linear growth. Moreover, any limit point of $\{u_\ve\}$ as
 $\ve\to0$ satisfies equation \eqref{eq:ODE} because $a$ is continuous.
Hence, if the limit $\lim_{\ve\to0} u_\ve$ (in distribution) exists,
 then this limit may be considered as a natural selection of a solution to \eqref{eq:ODE}.

The corresponding problem was originated in papers by  Bafico and Baldi
 \cite{Ba, BaficoBaldi}, who considered the one-dimensional case;
other generalizations see, for example, in
\cite{ BOQ, BS,   Flandoli, DFV, DelarueMaurelly, DLN,
 H, KrykunMakhno,   PilipenkoPavlyukevich, PilipenkoProske1, PilipenkoProskeSelfSimilar, Trevisan} and references therein. Investigations in multidimensional case are much complicated than   in   the one-dimensional   one.   There are still  no simple sufficient
conditions that ensure existence of  a limit  
$\lim_{\ve\to0} u_\ve$ and a
 characterization of this limit. One of the reason for this is the absence of the
 linear ordering in the
multidimensional case. Indeed, in the one-dimensional situation the are only   two   ways to exit from the point 0: one way to the right and another to the left.
 The probability of going left or right can be easily obtained since there are explicit formulas
for hitting probabilities for one-dimensional diffusions. The equation for the limit process
 outside of 0 must satisfy the original ODE because $a$ is   Lipschitz  continuous there.

In this paper we consider the multidimensional case, where the Lipcshitz  condition for $a$
may fail at a hyperplane. Let us describe the corresponding model.
Consider an SDE
\bel{eq:mainSDE_small}
\ba
{d u_\ve(t)}= a(u_\ve(t)) dt+ \ve \sigma(u_\ve(t))dW(t);\\
u_\ve(0)=x^0,
\ea
\ee
where $a:\mbR^d\to\mbR^d$, $\sigma:\mbR^d\to\mbR^{d\times m} $ are measurable functions, $W$ is an $m$-dimensional Wiener process.

Assume that  $a$ and $\sigma$ are  of linear growth, $\sigma $ is continuous and satisfies the uniform ellipticity condition.  This  ensures existence and uniqueness of a weak solution to \eqref{eq:mainSDE_small} and relative compactness for the distributions of $\{u_\ve\}$.


Set $H:=\mbR^{d-1}\times\{0\}$. Suppose that
the initial starting point $x^0\in H$
and that the drift
$a$ satisfies the local Lipschitz property in $\mbR^d\setminus H$.

Note that the definition of $a$ on $H$ is inessential because $u_\ve$ spends zero time in $H$ with probability 1 due to the non-degeneracy of the diffusion coefficient.

The case when $a$ is globally Lipschitz continuous in 
the lower half-space  $\mbR^d_-:=\mbR^{d-1}\times(-\infty,0)$ and
and the upper half-space $\mbR^d_+:=\mbR^{d-1}\times(0,\infty)$ was  investigated in \cite{PilipenkoProske1}. The result was formulated in terms of the vertical components of
$a^\pm(x^0):=\lim_{x\to x^0, x\in \mbR^d_\pm  } a(x).$ In this paper we investigate the case when the drift has H\"older-type asymptotic in a neighborhood of $H$. Namely, we will assume that

 {\bf A1.} {\it $a_d(x) = |x_d|^\gamma b(x),$ where  $\gamma<1$, $x_d$
 is the $d$-th coordinate of $x=(x_1,...,x_d)$, and $b$ is a globally  Lipschitz continuous function
in $\mbR^d_+$ and $\mbR^d_-,$  $b^\pm(x)\neq 0, x\in H$.}

 {\bf A2.} {\it $a_k, k=1,\dots,d-1,$ are  globally  Lipschitz functions in $\mbR^d_+$
 and $\mbR^d_-.$}

   This case has new features, and the proofs will be based on new ideas compared to the proofs  from
 \cite{PilipenkoProske1}.   To illustrate the difference,  let us recall
 briefly  results of \cite{PilipenkoProske1}, where the case  $\gamma=0$ was considered,
and sketch the expected results in the case $\gamma\in(0,1)$.

\textit{Case 1. (The vector field $a$ pushes outwards the hyperplane)}
  Denote by $\mathbf{n}=(0,...,0,1)$    the normal vector to the hyperplane $H.$  Assume that
$\gamma=0$ and  $\pm(a^\pm(x),\mathbf{n})>0$,  $x\in H$. Then there are    two  
solutions $u^\pm$
 to
  \bel{eq:mainODE}
{d u(t)}= a(u(t)) dt
\ee
 that start at $x^0\in H$ and exit  from  $H$ immediately to the upper and the lower half spaces, respectively.
It was proved in  \cite{PilipenkoProske1} that if $\gamma=0$, then
 the limit process $u_0$ immediately leaves $H$ and moves as $u^\pm$ with probabilities
proportional to $|(a^\pm(x^0),\mathbf{n})|$.
The corresponding  proof was similar to the one-dimensional situation. It used some
 comparison principle  adapted to the multidimensional situation.
 Investigations for arbitrary $\gamma\in(0,1)$ will be similar, but selection probabilities will be  different.
\begin{rem}
It was assumed in \cite{PilipenkoProske1} that the noise is additive, i.e.,
 $\sigma $ is the identity matrix and $m=d$.   The case of multiplicative noise is completely analogous.
\end{rem}
\begin{rem}
If  $\gamma=0$ and the vector field $a$ pushes away $H$ from one side of
 $H$ and attracts from another side (for example, $(a^\pm(x),\mathbf{n})>0$), then there is a unique solution to  \eqref{eq:mainODE} that starts at $x^0\in H$.
 This solution exits from $H$ immediately (to the upper half space in our case) and  the
 limit process $u_0$ equals this solution of the ODE, see \cite{PilipenkoProske1}.

If $\gamma\in(0,1), $ the result is similar. Assume, for example, that $b^\pm(x^0)>0$. Then there exists a unique solution to  \eqref{eq:mainODE} that exit $H$ immediately (there may be other solutions that stay in $H$). Moreover
this solution exits to the upper half space and the
limit process $u_0$ equals this solution. We do not prove this result in this paper. The proof is similar to  \cite{PilipenkoProske1}.
\end{rem}
\textit{Case 2. (The vector field $a$ pushes towards the hyperplane)}
 Assume that $\gamma=0$ and $\pm(a^\pm(x),\mathbf{n})<0$, $x\in H$.
 It can be seen that any limit point of $\{u_\ve\}$
 must stay at $H$ with probability 1.
  It was proved in \cite{PilipenkoProske1} that the limit process $u_0$ satisfies an ODE on $H$
 with the drift
 $P_H(p_+(x) a^+(x)+p_-(x)a^-(x)),$ where  $P_H$ is the orthogonal projection to $H$ and the
 coefficients $p_\pm(x)$
 are equal to $\frac{a_d^\mp(x)}{a^-_d(x)-a^+_d(x)}.$
Note that this multidimensional result has no one-dimensional analogues, where the limit
is zero process.
In multidimensional case the first $(d-1)$ coordinates may change while $d$-th coordinate stays zero.

The idea of proof was to analyze the time spent by $u_\ve$ in upper
 and lower half-spaces. It was seen that since any limit process stays at $H$ and $u_\ve$  is close
  close
 to $H$ for small $\ve$,  then the   times  spent in upper and lower half-spaces in a
neighborhood of $x\in H$   are   proportional
 to the $d$-th coordinates  $a^-_d(x)$ and $a^+_d(x)$,
 respectively (they are not zero if $\gamma=0$). Note that, the proof in  \cite{PilipenkoProske1}
 was independent
of the type of a noise. The small noise  might be arbitrary  process that    (a) ensures existence a solution
and (b)  converges   to 0 uniformly in probability as $\ve \to 0$ (however, the corresponding
results were formulated for Brownian noise only).

The proof from \cite{PilipenkoProske1}  does not work if $a_d(x)\to0$ as $x$  approaches to  $ H$. The time spent in upper and lower half-spaces
 might depend on the asymptotic of  decay of  $a_d$ in a neighborhood of $H.$
In this paper we prove the result when $a$ satisfies assumptions {\bf A1, A2} with $\gamma\in(0,1)$,
  $b^+(x)<0$ and $ b^-(x)>0$ for $ x\in H$.

It appears that if we  scale the vertical coordinate $\ve^{-\delta}u_{d,\ve}(t)$
 for a special choice of $\delta>0,$
 then a  pair
  $(u_{1,\ve}(t),..., u_{d-1,\ve}(t))$ and $\ve^{-\delta}u_{d,\ve}(t)$
can be considered  as   components of a Markov process in  a ``slow'' and ``fast'' time, respectively.
 Hence the description of the limit process for $\{u_\ve\}$ is closely related to the
averaging principle for Markov processes.
We will see that  the limit process satisfies an ODE on $H$
whose coefficients are an averaging of functions of $a_k^\pm, k=1,...,(d-1)$
over a stationary distribution of a scaled  vertical component given the other components were frozen.
The idea to use some scaling for small-noise problem
 was effectively used
in one-dimensional case if the drift is a power-type function and the noise is a Levy $\alpha$-stable
process or even more general.
\begin{rem}
The case $\gamma=1$ is critical. If $a_d(x)\sim x_d b(x),$ where $b^\pm(x)\neq0, x\in H$,
 then the limit process may be non-Markov and
  satisfy  certain equation \cite{PilipenkoProske1} that depends somehow on a  Wiener process $W$
(that formally should disappear in a limit equation).
\end{rem}

The paper is organized as follows. In section \ref{Main results1} we formulate the problem and the main results.
The proofs  for the cases when the drift pushes outwards $H$ and towards $H$ are given in  \S \ref{ProofTh2.1} and \S \ref{Th2.2}, respectively.

In subsection  \ref{Main results_averaging}  we
 also formulate an averaging principle, which is quite general and
   new result. The proof of averaging principle is postponed to section \ref{section:proof_AP}.

\noindent
\textbf{Acknowledgements.} The work of A.\ Kulik was supported by the Polish National Science Center grant 2019/33/B/ST1/02923.
 Research of A.\ Pilipenko was partially supported by Norway-Ukrainian cooperation in mathematical education
Eurasia 2016-Long-term 
CPEA-LT-2016/10139 and by  the Alexander von Humboldt Foundation within 
the Research Group Linkage Programme 
{\it Singular diffusions: analytic and stochastic approaches} between the University of Potsdam 
and the Institute of Mathematics of the National Academy of Sciences of Ukraine. 

 \section{Main results} \label{Main results1}
Let us represent $u_\ve(t)$
as a pair $(X_\ve(t), Y_\ve(t))$, where $Y_\ve$ is the
last coordinate of $u_\ve$ and $X_\ve$ consists of the first   $d-1$   coordinates. Below
we study only the general problem for the pair $(X_\ve(t), Y_\ve(t))$, which
can be easily be reformulated for $u_\ve.$ For   notational   convenience, we assume below that $X_\ve$
is a $d$-dimensional process but not $(d-1)$ dimensional one.

The general setup is the following. Let $X_\ve, Y_\ve$ be stochastic processes with
 values in $\mbR^d$ and $\mbR,$ respectively. Assume that the pair $X_\ve, Y_\ve$
 satisfies the following SDE
  \be\label{Small_Wiener}
\ba
\d X_\eps(t)&=\psi\big(X_\eps(t), Y_\eps(t)\big)\, \d t + \eps \,b\big(X_\eps(t), Y_\eps(t)\big) \d B(t),\\
\d Y_\eps(t)&= \phi (X_\eps(t), Y_\eps(t))  Y^\gamma_\eps(t) \, \d t+
 \eps \beta\big( X_\eps(t),Y_\eps(t)\big)\, \d W(t),\\
X_\eps(0)&=x^0, \ Y_\eps(0)=0,
\ea
\ee
 where $B,W$ are  Wiener processes (multidimensional and one-dimensional), that may be dependent.

 Denote
 \[
 y^\gamma:= |y|^\gamma (\1_{y>0} - \1_{y<0});
 \]
 \[
 H:=\mbR^d\times \{0\}.
 \]

Assume that

{\bf B1} $\psi(x,y)=\psi^+(x,y)\1_{y\geq0}+\psi^-(x,y)\1_{y<0}$ and
$\phi(x,y)=\phi^+(x,y)\1_{y\geq0}+\phi^-(x,y)\1_{y<0}$, where functions $\psi^\pm,$ $\phi^\pm$  are bounded, continuous in $x,y$.

 We assume that  domains of $\psi^\pm, \vf^\pm$
 are the whole space $x\in\mbR^d, y\in \mbR,$
 despite we use their values on the corresponding half-spaces only.  The functions $\psi, \vf$
 may have jump discontinuity on $H.$

\vskip10pt
{\bf B2} 
$\phi^\pm(x,0)\neq 0$ for any $x\in \mbR^d$ ;

{\bf B3}  $\beta(x,y)=\beta^+(x,y)\1_{y\geq0}+\beta^-(x,y)\1_{y<0}$, where $\beta^\pm$ are
 bounded, continuous and separated from zero function in the whole space $\mbR^d\times \mbR$;
function  $b$ is bounded and continuous in  $(\mbR^d\times \mbR)\setminus H$;


{\bf B4} $ \gamma\in(0,1). $ 


Under   assumptions {\bf B1}--{\bf B4} there exists a weak solution to \eqref{Small_Wiener}.

Indeed, it follows from the  standard compactness arguments that there exists a weak solution to
\[
\ba
\d \hat X_\eps(t)&=\frac{\psi}{\beta^2}\big(\hat X_\eps(t), \hat Y_\eps(t)\big)dt+\eps \,\frac{b}{\beta}\big(\hat X_\eps(t), \hat Y_\eps(t)\big) \d B(t),\\
\d \hat Y_\eps(t)&=
 \eps   \d W(t),\\
\hat X_\eps(0)&=x^0, \ \hat Y_\eps(0)=0.
\ea
\]
Note that all coefficients may be discontinuous in $H$ but the processes
 spend  zero time there with probability 1.  Any redefinition of coefficients in $H$
 does not affect the equations.

Using the  transformation of time arguments, see for example  \cite{IkedaWatanabe}, we get a solution to
\[
\ba
\d \hat{ \hat X}_\eps(t)&={\psi}\big(\hat{\hat X}_\eps(t), \hat{\hat Y}_\eps(t)\big)dt+\eps \,{b}\big(\hat{\hat X}_\eps(t), \hat{\hat Y}_\eps(t)\big) \d B(t),\\
\d \hat{\hat Y}_\eps(t)&=
 \eps   \beta\big(\hat{\hat X}_\eps(t), \hat{\hat Y}_\eps(t)\big)\d W(t),\\
\hat{\hat X}_\eps(0)&=x^0, \ \hat{\hat Y}_\eps(0)=0.
\ea
\]
Finally, Girsanov's theorem yields existence of a weak solution to \eqref{Small_Wiener} .

\begin{rem} If $b$ is non-degenerate, then existence of a solution
can be proved  without transformation of time arguments.
\end{rem}

\subsection{Repulsion from the hyperplane}\label{subsection:repulsion}
In this subsection we assume that $\phi^\pm(x,0)>0$
for all $x\in\mbR^d.$ Then $\sgn(y) \vf(x,y)y^\gamma>0, \ y\neq 0$ and the drift pushes away from the hyperplane $\mbR^d\times\{0\}.$

Suppose that assumptions {\bf B1}--{\bf B4} holds true and 
 functions
$\psi^\pm, \vf^\pm$ are  locally Lipschitz continuous
in $(x,y)\in\mbR^d\times\mbR$.

Then there are unique solutions $(X^+(t), Y^+(t))$ and $(X^-(t), Y^-(t))$ to the unperturbed system (i.e., $\ve=0$):
\[
\ba
\d X(t)&=\psi\big(X(t), Y(t)\big)\, \d t ,\\
\d Y(t)&= \phi (X(t), Y(t))  Y^\gamma(t) \, \d t,\\
X(0)&=x^0, \ Y(0)=0,
\ea
\]
such that $ Y^+(t)>0$ and $ Y^-(t)<0$ for all $t>0.$

Indeed, set $\tilde Y(t):=Y^{1-\gamma}(t).$ Then
\[
\ba
 X(t)&=x^0+\int_0^t\psi\big(X(s), \tilde Y^{\frac{1}{1-\gamma}}(s)\big)\, \d s ,\\
\tilde Y(t)&= (1-\gamma)\int_0^t\phi\big(X(s), \tilde Y^{\frac{1}{1-\gamma}}(s)\big)\, \d s.
\ea
\]
Since $\gamma\geq 0,$ the functions $(x,\tilde y)\to \psi^\pm(x,\tilde y)$
and $(x,\tilde y)\to \phi^\pm(x,\tilde y)$ are locally Lipschitz continuous.
So, equations
\[
\ba
 X^\pm(t)&=x^0+\int_0^t\psi^\pm\big(X^\pm(s), (\tilde Y^\pm(s))^{\frac{1}{1-\gamma}}\big)\, \d s ,\\
\tilde Y^\pm(t)&=(1-\gamma) \int_0^t\phi^\pm\big(X^\pm(s), (\tilde Y^\pm(s))^{\frac{1}{1-\gamma}}\big)\, \d s.
\ea
\]
have unique solutions $(X^\pm(t), \tilde Y^\pm(t))$  and these solutions are such that
$\tilde Y^+(t)>0$ and $ \tilde Y^-(t)<0$ for all $t>0.$ Making the inverse change of variables
we get  the desired functions $Y^\pm(t)=(\tilde Y^\pm(t))^{\frac{1}{1-\gamma}}$.

The solution does not explode in a finite time because $\psi^\pm, \vf^\pm$ are bounded by assumption
 {\bf B1}.


\begin{thm}\label{thm:main_small_noise_rep}
The distribution of $(X_\ve, Y_\ve)$ in $C([0,T])$  converges weakly as $\ve\to 0$ to the  measure
$$
p_-\delta_{(X^-, Y^-) }+p_+\delta_{(X^+, Y^+) }
$$
where

\bel{eq_p+-}
p_\pm=
\frac{\left(\frac{  \vf^+(x^0,0)  }{ (\beta^+(x^0,0))^2  }\right)^{\frac{1}{\gamma+1}}}
{\left(\frac{  \vf^-(x^0,0)   }{ (\beta^-(x^0,0))^2  }\right)^{\frac{1}{\gamma+1}}+ \left(\frac{  \vf^+(x^0,0)   }{ (\beta^+(x^0,0))^2  }\right)^{\frac{1}{\gamma+1}}}
\ee
 and $\delta_{(X^+, Y^+) }$, $\delta_{(X^-, Y^-) }$  means the unit mass that concentrated on the
 functions
 $(X^+, Y^+)$ and $(X^-, Y^-)$, respectively.
\end{thm}
The proof is given in \S  \ref{ProofTh2.1}.

\begin{rem}
If $\pm\phi^\pm(x,0)>0$ (or $\pm\phi^\pm(x,0)<0$) for all $x\in\mbR^d$, then the limit process is $(X^+(t), Y^+(t))$ (respectively $(X^-(t), Y^-(t))$ ) with probability 1.
\end{rem}
\begin{rem}
If we have inequality $\phi^+(x^0,0)>0$ and $\phi^-(x^0,0)<0$ only at the initial point (and hence in some neighborhood by continuity of coefficients),
 then the
functions $(X^\pm(t), Y^\pm(t))$ are well defined up to the moment $\tau_H^{\pm}:=\inf\{ t>0\ : \ Y^\pm(t)=0\}$ of the first return to $H.$ In this case we have the convergence in distribution for the stopped processes:
\[
(X_\ve(\cdot\wedge \tau^+_H\wedge \tau^-_H), Y_\ve(\cdot\wedge \tau^+_H\wedge \tau^-_H))\Rightarrow
p_-\delta_{(X^-(\cdot\wedge \tau^+_H\wedge \tau^-_H), Y^-(\cdot\wedge \tau^+_H\wedge \tau^-_H)) }+p_+\delta_{(X^+(\cdot\wedge \tau^+_H\wedge \tau^-_H), Y^+(\cdot\wedge \tau^+_H\wedge \tau^-_H)) }.
\]
The proof is   essentially   the same,  but it   involves    routine localization arguments   in addition.
\end{rem}

\subsection{Attraction to the hyperplane}
In this subsection we assume that $\phi^\pm(x,0)<0$ 
 for all $x\in\mbR^d.$ 

Suppose that assumptions {\bf B1}--{\bf B4} holds true  and $\psi^\pm $ are  locally Lipschitz in $x$ for any fixed $y$.

\begin{thm}\label{thm:main_small_noise_attr}
For any $T>0$ we have the uniform convergence in probability
\[
\lim_{\ve\to0}\sup_{t\in[0,T]}\|(X_\ve(t), Y_\ve(t))- (X(t),0)\|=0,
\]
where $X(t)$ is a solution to the following  ODE
\[
dX(t) = \bar\psi(X(t))dt,\ \ \ X(0)=0,
\]
and
\bel{eq:aver_coeff}
\overline{\psi}(x)=\psi^+(x,0)
\frac{\left( \frac{(\beta^+(x,0))^2}{\phi^+(x,0)} \right)^{ \frac{1}{\gamma+1}}}
{\left( \frac{(\beta^-(x,0))^2}{\phi^-(x,0)} \right)^{ \frac{1}{\gamma+1}}+
\left( \frac{(\beta^+(x,0))^2}{\phi^+(x,0)} \right)^{ \frac{1}{\gamma+1}}} +
 \psi^-(x,0)
\frac{\left( \frac{(\beta^-(x,0))^2}{\phi^-(x,0)} \right)^{ \frac{1}{\gamma+1}}}
{\left( \frac{(\beta^-(x,0))^2}{\phi^-(x,0)} \right)^{ \frac{1}{\gamma+1}}+
\left( \frac{(\beta^+(x,0))^2}{\phi^+(x,0)} \right)^{ \frac{1}{\gamma+1}}}
,
\ee
\end{thm}
The proof is given in \S \ref{Th2.2}.
\begin{rem}
Note that
\[
\frac{{\phi^+(x,0)}^{-\frac{1}{\gamma+1}}}{{\phi^+(x,0)}^{-\frac{1}{\gamma+1}}
+ {\phi^-(x,0)}^{-\frac{1}{\gamma+1}}}
=\pi^{(x)}([0,\infty))
,\ \ \
\frac{{\phi^-(x,0)}^{-\frac{1}{\gamma+1}}}{{\phi^+(x,0)}^{-\frac{1}{\gamma+1}}+ {\phi^-(x,0)}^{-\frac{1}{\gamma+1}}}=\pi^{(x)}((-\infty,0)),
\]
where $\pi^{(x)}$ is the stationary distribution for  the SDE
\[
\d y^{(x)}(t) =
 (\phi^+(x, 0)\1_{y^{(x)}(t)>0} + \phi^-(x, 0)\1_{y^{(x)}(t)<0})(y^{(x)}(t))^\gamma  \, \d  t+
   \beta\big( x,0) \, \d W(t).
\]
Hence,
\[
\overline{\psi}(x)=\psi^+(x,0) \pi^{(x)}([0,\infty))+ \psi^-(x,0) \pi^{(x)}((-\infty,0)),
\]
i.e., the drift of the limit equation is the averaging of $\psi^\pm$
over the stationary distribution of an SDE  with frozen $x$ variable.
The corresponding relation between the averaging principle and  averaging of coefficients in
  the limit equation for the small noise perturbation problem  will be seen from the proof.
\end{rem}

In the next subsection   we formulate an averaging principle, which  is applied in the proof
 of Theorem \ref{thm:main_small_noise_attr}.
We consider more general SDEs than \eqref{Small_Wiener} because
the idea of the proof is universal. The corresponding result may be interesting by itself.


  \subsection{Averaging} \label{Main results_averaging}

Let  for $\eps>0$ the processes $X_\eps(t), Y_\eps(t)$ take values in $\mbR^d, \mbR^k$ and have the form
\be\label{DA_fully coupled}
\ba
X_\eps(t)&=X_\eps(0)+\int_0^t a^\eps\big(X_\eps(s), Y_\eps(t)\big)\, \d s+
\int_0^t\sigma^\eps\big(X_\eps(s), Y_\eps(s)\big)\, \d B^\eps_s
\\&+ \int_0^t\int_{\mbR^m} c^\eps\big(X_\eps(s-), Y_\eps(s-),u\big)\Big[N^\eps(\d u,\d s)-1_{|u|\leq \rho}\nu^\eps(\d u)ds\Big]+\xi_\eps(t),
\\
Y_\eps(t)&=Y_\eps(0)+\eps^{-1}\int_0^tA^\eps\big(X_\eps(s), Y_\eps(s)\big)\, \d s+\eps^{-1/2}\int_0^t\Sigma^\eps\big( X_\eps(s),Y_\eps(s)\big)\, \d W_{s}^\eps
\\&+\int_{0}^{t} \int_{\mbR^l} C^\eps\big(X_\eps(s-), Y_\eps(s-),z\big)\Big[Q^\eps(\d z,\d s)-1_{|z|\leq \rho}\eps^{-1}\mu^\eps(\d z)\d s\Big],
\ea
\ee
where $B^\eps_t, W^\eps_t$ are Brownian motions and $N^\eps(\d u,\d t), Q^\eps(\d z,\d t)$ are
 Poisson point measures on a common filtered probability space $(\Omega^\eps, \cF^\eps, \P^\eps)$,
 and the random measures $N^\eps(\d u,\d t)$, $Q^\eps(\d z,\d t)$ have the intensity  measures
$\nu^\eps(\d u)\d t$ and $\eps^{-1}\mu^\eps(\d z)\d t$, respectively. These random measures are
 involved into the system in the partially compensated form, which is quite typical for the
 L\'evy-driven SDEs; what is a bit unusual is the  choice of the cutoff functions
$1_{|u|\leq \rho}, 1_{|z|\leq \rho}$ with the number $\rho>0$ to be specified separately.
 This choice will become clear later, when we describe the limit behavior of the L\'evy measures
 $\nu^\eps(\d u), \mu^\eps(\d z)$ as $\eps\to 0$.  Note that here and below we do not assume a uniqueness
of a solution to prelimit equation \eqref{DA_fully coupled}.

The factor $\eps^{-1}$ in the intensity measure for $Q^\eps(\d z,\d t)$ and the factors $\eps^{-1}, \eps^{-1/2}$ at the integrals w.r.t.
 $\d s$ and $\d W^\eps_s$ in the equation for   $Y_\eps$ mean that the evolution of the component $Y_\eps$ happens at the `fast' time scale $\eps^{-1} t$, which we will also call the `microscopic' time scale.
 The component $X_\eps$ evolves at the `slow', or `macroscopic' time scale $t$; its evolution involves the deterministic term, two stochastic terms (continuous and partially compensated jump parts), and a residual term $\xi_\eps$, for which  we do not impose any structural assumptions, and only require it to be asymptotically small in the following sense:

\noindent$\mathbf{H}_0.$ (Negligibility of the residual term). The process $\xi_\ve(t)$ is an adapted c\`adl\`ag process, and for any $T>0$,
$$
\sup_{t\in [0,T]}|\xi_\eps(t)|\to 0, \quad \eps\to 0
$$
in probability.

The aim of this subsection is to get the averaging principle (AP) for the `slow' component $X_\eps$. Let us stress that the framework we adopt is quite general; in particular,
\begin{itemize}
  \item the two-scale system \eqref{DA_fully coupled} is \emph{fully coupled} in the sense that the coefficients of the `slow' component depend on the `fast' one, and vice versa;
  \item the noises for the `slow' and the `fast' component are allowed to be dependent;
  \item the coefficients of the `slow' component can be discontinuous.
\end{itemize}
Let us introduce further assumptions on the system \eqref{DA_fully coupled}.
Note that all the assumptions listed below are quite natural and non-restrictive.

\noindent$\mathbf{H}_1.$ (Bounds for the  coefficients). There exists a constant $C$ such that
$$
|a^\eps(x,y)|\leq C, \quad |\sigma^\eps(x,y)|\leq C, \quad |\Sigma^\eps(x,y)|\leq C, \quad |c^\eps(x,y,u)|\leq C |u|, \quad |C^\eps(x,y,z)|\leq C |z|
$$
for all values of $x,y,u,z.$

In addition, for any $R>0$
there exists a constant $C_R$  such that
$$
|A^\eps(x,y)|\leq C_R, \quad x\in \mbR^d, \quad |y|\leq R.
$$

\noindent$\mathbf{H}_2.$ (Bounds for the L\'evy measures). There exist constants $C$ and $p>0$ such that
$$
\int_{\mbR^m}(|u^2|\wedge 1)\nu^\eps(\d u)\leq C, \quad \int_{\mbR^l}(|z^2|1_{|z|\leq 1}+|z|^p1_{|z|>1})\mu^\eps(\d z)<\infty.
$$

\noindent$\mathbf{H}_3.$ (The coefficients of the fast component are convergent). There exist continuous functions $A(x,y), \Sigma(x,y), C(x,y,z)$ such that
$$
A^\eps(x,y)\to A(x,y), \quad  \Sigma^\eps(x,y)\to  \Sigma(x,y), \quad   \mbox{and} \quad  C^\eps(x,y,z)\to C(x,y,z)  \quad  \mbox{as} \quad \eps\to 0
$$
uniformly on every compact set in $\mbR^d\times \mbR^k,\mbR^d\times \mbR^k$,
and $\mbR^d\times \mbR^k\times(\mbR^l\setminus\{0\})$, respectively.

To introduce the next condition, let us define the \emph{weak convergence} of a family of L\'evy measures on $\mbR^m$ in the following way: $\nu^\eps(\d u)\Longrightarrow \nu(\d u)$
 if for every
continuous function $\phi$ with a  support compactly embedded into $\mbR^m\setminus\{0\}$,
$$
\int_{\mbR^m}\phi(z)\, \nu^\eps(\d z)\to \int_{\mbR^m}\phi(z)\, \nu(\d z), \quad \eps\to 0.
$$

\noindent$\mathbf{H}_4.$ (The L\'evy measures of the noises are weakly convergent). There exist L\'evy measures $\nu(\d u),$ $ \mu(\d z)$ on $\mbR^m, \mbR^l$ respectively such that
$$
\nu^\eps(\d u)\Longrightarrow \nu(\d u) \quad  \mbox{and} \quad \mu^\eps(\d z)\Longrightarrow \mu(\d z)
\quad   \mbox{as} \quad \eps\to 0.
$$
In addition,
\be\label{rho_cont}
\nu(\{u:|u|=\rho\})=0, \quad \mu(\{z:|z|=\rho\})=0.
\ee

Condition \eqref{rho_cont} yield that the cutoff functions $1_{|u|\leq \rho}, 1_{|z|\leq \rho}$ used in \eqref{DA_fully coupled} are a.s. continuous w.r.t. the measures $\nu(\d u), \mu(\d z)$, respectively. Note that there exists at most countable set of levels $\rho$ such that \eqref{rho_cont} fails, hence one can always choose $\rho$ to satisfy this condition. Of course, changing the cutoff level would change the drift coefficients respectively.

 Next, assume  that the drift of the fast component performs an attraction to origin.

\noindent$\mathbf{H}_5.$ (The drift condition for the microscopic dynamics) There exist $\kappa>0$ and $c,r>0$ such that
\be\label{drift_f}
A^\eps(x,y)\cdot y\leq -c|y|^{\kappa+1}, \quad |y|\geq r.
\ee
In addition, the \emph{balance condition} holds:
\be\label{balance}
\kappa+p>1,
\ee
where $p$ is introduced in the assumption $\mathbf{H}_2.$

Consider a family of `frozen microscopic equations'
\be\label{frozen}
\d y(t)=
A \big(x, y(t)\big)\, \d t+\Sigma\big( x,y(t-)\big)\, \d W_{ t}+
 \int_{\mbR^l} C\big(x, y(t-),z\big)\, \Big[Q(\d z,\d s)-1_{|z|\leq 1}\mu(\d z)\d s\Big],
\ee
where $W$ is a  Wiener process and $Q(\d z,\d t)$ is an independent Poisson point measure with the intensity measure $\mu(\d z)\d t$. For the corresponding `frozen dynamics' we introduce a separate family of assumptions.

\noindent $\mathbf{F}_0.$ (The `frozen microscopic dynamics' is well defined and Feller).
 For any $x$ and any initial value $y(0)=y$, the SDE \eqref{frozen} has  a unique weak solution,
which is a  Markov process.
Furthermore we denote the corresponding family of Markov processes by $y^{(x)}, x\in \Re^d$,
and write $P_{t}^{(x)}(y,\d y')$ for the corresponding family of transition probabilities.

  We also denote
$$
P_t^{frozen}f(x,y)=\int_{\mbR^k} f(x,y')P_t^{(x)}(y, \d y'), \quad t\geq 0,
$$
the semigroup of operators corresponding to the two-component process $(x,y^{(x)})$ in which the first component is constant and the second one is the Markov process specified above. We assume that
this semigroup  is Feller.

 For this family,
we assume the following mixing property, which is actually the local Dobrushin condition, uniform in parameter $x$; see \cite[Section~2]{K_book}.

\noindent$\mathbf{F}_1.$ (The `frozen microscopic dynamics' is locally mixing). There exists $h>0$ such that, for any $R>0$ there exists $\rho=\rho_R>0$ such that, for any $x,y_1,y_2$ with $|x|\leq R, |y_1|\leq R, |y_2|\leq R$
$$
\|P_{h}^{(x)}(y_1,\d y')-P_h^{(x)}(y_2, \d y')\|_{TV}\leq 1-\rho,
$$
where $P_t^{(x)}(y, \d y')$ denotes the transition probability of the process $y^{(x)}$, and \emph{the total variation} distance between probability measures is defined as
$$
\|\lambda_1-\lambda_2\|_{TV}=\sup_{A}(\lambda_1(A)-\lambda_2(A)).
$$

We note that assumptions $\mathbf{F}_1,$ $\mathbf{H}_5$ ensure that, for each $x\in \mbR^d$,  the laws of $y^{(x)}_t$ converge to the invariant probability measure (IPM)  $\pi^{(x)}(\d y)$ with an explicitly rate; see Proposition \ref{p1} below.

For the coefficients of the `slow' component, we assume a weaker analogue of $\mathbf{H}_3,$ where the convergence and continuity of the limiting coefficients may fail on an exceptional set, which should be negligible, in a sense.

\noindent$\mathbf{H}_6.$ (The coefficients of the slow component are convergent). There exist functions $a(x,y)$, $\sigma(x,y)$, $c(x,y,u)$ and  an open set $B\subset \mbR^d\times\mbR^k$ such that,
for any compact set $K\subset B$,
$$
a^\eps(x,y)\to a(x,y)  \quad  \mbox{and} \quad  \sigma^\eps(x,y)\to \sigma(x,y)  \quad  \mbox{as} \quad  \eps\to 0
$$
uniformly on $K$, and for any $R>1$
$$
c^\eps(x,y,u)\to c(x,y,u), \quad \eps\to 0
$$
uniformly on $K\times\{u:R^{-1}\leq |u|\leq R\}$. The set $\Delta=(\mbR^d\times\mbR^k)\setminus B$ satisfies
$$
\pi^{(x)}\{y:(x,y)\in \Delta\}=0 \ \ \mbox{for any}\ x\in\mbR^d.
$$
In addition, the functions $a(x,y)$, $\sigma(x,y)$, and $c(x,y,u)$ are continuous on $B$ and $B\times (\mbR^m\setminus\{0\})$, respectively.

Define the averaging of the limiting drift coefficient for the macroscopic component w.r.t. the family of IPMs for the frozen microscopic one:
$$
\overline{a}(x)=\int_{\Re^k}a(x,y)\pi^{(x)}(\d y).
$$
Next, consider the limiting diffusion matrix and compensated/non-compensated jump kernels for the macroscopic component,
$$
b(x,y)=\sigma(x,y)\sigma(x,y)^*, \quad K_{(\rho)}(x,y, A)=\nu(\{u: |u|\leq \rho,  c(x,y,u)\in A\}),
$$
$$K^{(\rho)}(x,y, A)=\nu(\{u: |u|> \rho,  c(x,y,u)\in A\}),
$$
and introduce the corresponding averaged characteristics as
$$
\overline{b}(x)=\int_{\Re^k}b(x,y)\pi^{(x)}(\d y), \quad \overline{K}_{(\rho)}(x, \d v)=\int_{\Re^k}K_{(\rho)}(x,y, \d v)\pi^{(x)}(\d y),
$$
$$
\overline{K}^{(\rho)}(x, \d v)=\int_{\Re^k}K^{(\rho)}(x,y, \d v)\pi^{(x)}(\d y).
$$

Finally, we introduce an auxiliary technical assumption.

\noindent$\mathbf{A}_0.$  The averaged coefficients $\overline{a}(x)$, $\overline{b}(x)$ are continuous. The averaged L\'evy kernels  $\overline{K}_{(\rho)}(x, \d v)$, $\overline{K}^{(\rho)}(x, \d v)$ depend on $x$ continuously, in the sense that
$$
\overline{K}_{(\rho)}(x', \d v)\Longrightarrow \overline{K}_{(\rho)}(x, \d v)  \quad
 \mbox{and} \quad  \overline{K}^{(\rho)}(x', \d v)\Longrightarrow \overline{K}^{(\rho)}(x, \d v) \quad  \mbox{as} \quad  x'\to x.
$$

\begin{rem}
It is  easy to give a sufficient condition for $\mathbf{A}_0$ to hold.  Namely, it is enough to assume, in addition to  $\mathbf{H}_0 - \mathbf{H}_6,$ $\mathbf{F}_0, \mathbf{F}_1,$ that the transition probabilities $P_{t}^{(x)}(y,\d y')$ are continuous in $x$ w.r.t. the total variation convergence for each $y\in \mbR^k, t\geq t_0$.  Then, because of the convergence \eqref{convTV}, the same continuity holds for the family of the IPMs $\pi^{(x)}(\d y)$. The latter continuity, combined with $\mathbf{H}_1, \mathbf{H}_2$, $\mathbf{H}_4$, and $\mathbf{H}_6,$  yields the required continuity of the  averaged coefficients.
\end{rem}

Now we are ready to formulate our main statement.

\begin{thm}\label{thm:mainAP} Assume $\mathbf{H}_0 - \mathbf{H}_6,$ $\mathbf{F}_0,\mathbf{F}_1,$ and
 $\mathbf{A}_0$ to hold,
$$X_\eps(0)\to x^0,\quad \eps\to0,$$
in probability and $\{Y_\eps(0)\}$ be bounded in probability.

Then the family $\{X_\eps, \eps>0\}$ is weakly compact in $\mathbb{D}([0, \infty), \mathbb{R}^d)$, and any
of  its weak limit point as $\eps\to0$ is a solution to the  martingale problem $(L, C_0^\infty)$ with 
\be\label{MP}
\ba
L\phi(x)&=\nabla \phi(x)\cdot \overline{a}(x)+\frac{1}{2}\nabla^2 \phi(x)\cdot  \overline{b}(x)
+\int_{\mbR^m}\Big( \phi(x+v)-\phi(x)-\nabla \phi(x)\cdot v\Big)\overline{K}_{(\rho)}(x, \d v)\\
&+\int_{\mbR^m}\Big( \phi(x+v)-\phi(x)\Big)\overline{K}^{(\rho)}(x, \d v)\\
&=\nabla \phi(x)\cdot \overline{a}(x)+\frac{1}{2}\nabla^2 \phi(x)\cdot  \overline{b}(x)
+\\
&+\int_{\mbR^k} \int_{\mbR^m} \Big( \phi(x+c(x,y,u))-\phi(x)-\nabla \phi(x)\cdot c(x,y,u) \1_{|u|\leq \rho}\Big)
\nu(du)\pi^{(x)}(\d y),
\ea
\ee
where $\phi\in C_0^\infty.$

If the martingale problem \eqref{MP} is well posed, then $X_\eps$ weakly converges as $\eps \to0$ to its unique solution with $X(0)=x^0$.
\end{thm}

\section{Proof of Theorem \ref{thm:main_small_noise_rep}}\label{ProofTh2.1}

The proof almost copying the proof of Theorem 3.1 in \cite{PilipenkoProske1}. Thus we only sketch
the main steps of the proof.

{\bf Step 1.} The sequence $\{(X_\ve, Y_\ve)\}$ is weakly relatively compact.
The proof follows from boundedness of functions $\phi, \psi, b, \beta.$

Therefore, to prove the Theorem it suffices to verify that any subsequence
$\{(X_{\ve_n}, Y_{\ve_n})\}$ contains sub-subsequence $\{(X_{\ve_{n_k}}, Y_{\ve_{n_k}})\}$
 that converges to the desired limit. Without loss of generality
 we will assume that  $\{(X_\ve, Y_\ve)\}$ is weakly convergent by itself.

{\bf Step 2.} Estimate for the time spent by $Y_\ve$ in a neighborhood of 0.

   We will use the following  general  statement.
\begin{lem}\label{lem:est_exit_time}
Assume that  processes $\{\eta_\ve(t)\}$ satisfy the following SDE
\[
\ba
d\eta_\ve(t) =a_\ve(t)\eta_\ve^\gamma(t) dt + \ve b_\ve(t) dW(t),\\
\eta_\ve(0)=0,
\ea
\]
where $|\gamma|<1$, and $a_\ve(t), b_\ve(t)$ are $\cF_t$-adapted processes such that
\[
a_\ve(t)\geq A>0,\ \ 0<C_1\leq b_\ve(t)\leq C_2
\]
for all $\omega, t, \ve.$

Set
\[
\tau_{\ve}(\delta):=\inf\{ t\geq 0\ : \ |\eta_\ve(t)|\geq \delta \}.
\]
Then there is a constant $K=K(A, C_1, C_2)$ such that
\[
\forall\delta>0\ \exists \ve_0>0\ \forall \ve\in (0,\ve_0):
\ \ \E\tau_\ve(\delta)\leq K \delta^{1-\gamma}.
\]
\end{lem}
The proof of Lemma is quite standard. We postpone it to the Appendix.

Without loss of generality we will assume that
\bel{eq:ass764}
\psi^\pm(x,0)\geq c_1>0  \mbox{ and }\ 0<c_2\leq\beta(x)\leq c_3 \ \mbox{ for all}\   x\in\mbR^d,
\ee
where $c_{1,2,3} $ are  some positive constants.
  This assumption does not restrict generality,  since  the general case can be considered using a localization.   Under this additional assumption,
 Lemma  \ref{lem:est_exit_time} applied to
\[
\tau_{\ve}(\delta):=\inf\{ t\geq 0\ : \ |Y_\ve(t)|\geq \delta \}.
\]
and the Chebyshev inequality yield
\bel{eq:estim_time}
\forall\delta>0\ \exists \ve_0>0\ \forall \ve\in (0,\ve_0):\ \ \
\P({\tau_\ve(\delta)}\geq \delta^{\frac{{1-\gamma}}{2 }})\leq K \delta^{\frac{{1-\gamma}}{2}}.
\ee
\begin{rem}
It can be seen from the construction of $Y^\pm$ that the inequality \eqref{eq:estim_time} is valid
for $\tau^\pm(\delta):=\inf\{ t\geq 0\ : \ |Y^\pm(t)|\geq \delta \}$ also.
\end{rem}
{\bf Step 3.}
We see from \eqref{eq:estim_time} that with high probability the
random variable $\tau_\ve(\delta) $ is dominated by $\delta^{\frac{{1-\gamma}}{2 }}.$
It follows from the  standard estimates for moments of SDEs that for small $t$ we have
$$
\E\sup_{s\in [0,t ]}|X_\ve(s)-x^0|^2\leq C t,
$$
where constant $C$ can be selected independently
 of $\ve\in[0,1]$.

So, we have the following estimates
\bel{eq:estim_dist}
\exists C_1>0\ \ \forall\delta>0\ \exists \ve_0>0\ \forall \ve\in (0,\ve_0):\ \ \
\P(\sup_{t\in [0,\tau_\ve(\delta) ]}|X_\ve(t)-x^0| \geq \delta^{\frac{{1-\gamma}}{6 }} ) \leq C_1\delta^{\frac{{1-\gamma}}{6}},
\ee
\bel{eq:estim_dist1}
\P(\sup_{t\in [0,\tau_\ve(\delta) ]}|X^\pm(t)-x^0|+|Y^\pm(t)|\geq 2\delta^{\frac{{1-\gamma}}{6 }} )
 \leq C_1 \delta^{\frac{{1-\gamma}}{6 }}.
\ee
Verify, for example, \eqref{eq:estim_dist}:
\[
\P(\sup_{t\in [0,\tau_\ve(\delta) ]}|X_\ve(t)-x^0| \geq \delta^{\frac{{1-\gamma}}{6 }} ) \leq
\P( \tau_\ve(\delta)> \delta^{\frac{{1-\gamma}}{2 }})+ \P(\sup_{t\in [0,\delta^{\frac{{1-\gamma}}{2 }}]}|X_\ve(t)-x^0| \geq \delta^{\frac{{1-\gamma}}{6 }} )
\leq
\]
\[
\delta^{\frac{{1-\gamma}}{2 }}+ \frac{\E \sup_{t\in [0,\delta^{\frac{{1-\gamma}}{2 }}]}|X_\ve(t)-x^0|^2 }{\delta^{\frac{{1-\gamma}}{3 }}}\leq
\delta^{\frac{{1-\gamma}}{2 }}+ \frac{C \delta^{\frac{{1-\gamma}}{2 }}}{\delta^{\frac{{1-\gamma}}{3 }}} \leq C_1\delta^{\frac{{1-\gamma}}{6}}.
\]

Note  also that
\be\label{eq:1072}
\sup_{t\in [0,\tau_\ve(\delta) ]} |Y_\ve(t)| =
|Y_\ve(\tau_\ve(\delta) )|=\delta \ \ \mbox{ a.s.}
\ee
 by the definition of $\tau_\ve(\delta).$

{\bf Step 4.} We denote by $(X^{x,y}(t), Y^{x,y}(t))$ a solution to the
 corresponding ODE that starts from $x\in \mbR^d,y\neq 0$.
This solution never hits $\mbR^d\times \{0\}, $ recall \eqref{eq:ass764}.
We  have correctness of  the definition of $(X^{x,y}(t), Y^{x,y}(t))$
 because in all other points
coefficients satisfy  the local Lipschitz condition.

If we wish to highlight that  $y>0$ (or  $y<0$), then the corresponding solution
  is denoted by $(X^{+,x,y}(t), Y^{+,x,y}(t))$ (or $(X^{-,x,y}(t), Y^{-,x,y}(t))$, respectively).

Let $\omega$ be such that $ Y_\ve(\tau_\ve(\delta))=\delta$, i.e., the process
$Y_\ve$ hits $\delta$ earlier than $-\delta.$ Then for this $\omega $ we have

\[
\sup_{t\in[0,T]} \left(|X_\ve(t)-X^+(t)|+|Y_\ve(t)-Y^+(t)|\right)
\leq
\]
\[
\sup_{t\in[0,T]} \left(|X_\ve(\tau_\ve(\delta)+t)-X^{+,X_\ve(\tau_\ve(\delta)),Y_\ve(\tau_\ve(\delta))}(t)|
+|Y_\ve(\tau_\ve(\delta)+t)-Y^{+,X_\ve(\tau_\ve(\delta)),Y_\ve(\tau_\ve(\delta))}(t)|\right)
+
\]
\[
\sup_{t\in[0,T]} \left(|X^{+,X_\ve(\tau_\ve(\delta)),Y_\ve(\tau_\ve(\delta))}(t)-X^+(\tau_\ve(\delta)+t)|+
|Y^{+,X_\ve(\tau_\ve(\delta)),Y_\ve(\tau_\ve(\delta))}(t)-Y^+(\tau_\ve(\delta)+t)|\right)
+
\]
\[
\sup_{t\in[0,T]} \left(|X^+(\tau_\ve(\delta)+t)-X^+(t)|+
|Y^+(\tau_\ve(\delta)+t)-Y^+( t)|\right)+ \sup_{t\in[0,\tau_\ve(\delta)]} (|X_\ve(t)-x^0|+|Y_\ve(t)| )=
\]
\[
I_1+...+I_4.
\]
Select small $\delta>0$ and after that  select $\ve_0>0$ from \eqref{eq:estim_time}. It follows from
\eqref{eq:estim_dist}, \eqref{eq:estim_dist1},
and construction of
$(X^+,Y^+)$ in \S\ref{subsection:repulsion} that $I_2, I_3, I_4$ are small with high probability.

To estimate $I_1$ we need the following statement  on integral equations. Let $f(t)=(f_X(t), f_Y(t))$
be a non-random continuous function,
and functions $X^{\pm, x,y}_{(f)}, Y^{\pm, x,y}_{(f)}  $ satisfy the integral equation
\[
\ba
 X^{\pm, x,y}_{(f)}(t)&=x+\int_0^t\psi^\pm\big(X^{\pm, x,y}_{(f)}(s), Y^{\pm, x,y}_{(f)}(s)\big)\, \d s +f_X(t) ,\\
 Y^{\pm, x,y}_{(f)}(t)&= \int_0^t \phi^\pm (X^{\pm, x,y}_{(f)}(s), Y^{\pm, x,y}_{(f)}(s))  (Y^{\pm, x,y}_{(f)})^\gamma(s) \, \d s+f_Y(t), \ t\in[0,T],\\
X^{\pm, x,y}_{(f)}(0)&=x, \ Y^{\pm, x,y}_{(f)}(0)=y.
\ea
\]
\begin{rem}
We do not assume that a pair $X^{\pm, x,y}_{(f)}, Y^{\pm, x,y}_{(f)}  $ is a unique solution.
 Recall also that the domains of $\psi^\pm, \vf^\pm$ is the whole space.
\end{rem}
\begin{lem}\label{lem:int_etim1}
\[
\ba
\forall  \delta>0\ \forall R\geq 1 \ \exists \alpha>0\  \forall x\in[-R,R] \ \forall t\in[0,T]\
\forall f, \ \|f\|_\infty<\alpha
\\
\forall y\in [\frac{1}{R}, R]:\ \ \  |X^{+, x,y}_{(f)}(t)- X^{+, x,y}(t)| +
 |Y^{+, x,y}_{(f)}(t)- Y^{+, x,y}(t)|\leq \delta
\\
\forall y \in [-R, -\frac{1}{R}]:\ \ \  |X^{-, x,y}_{(f)}(t)- X^{-, x,y}(t)| +
 |Y^{-, x,y}_{(f)}(t)- Y^{-, x,y}(t)|\leq \delta.
\ea
\]
\end{lem}
The proof of the Lemma is standard. Notice that if $\alpha$ is small
enough, then  $ Y^{\pm, x,y}_{(f)}(t)\neq 0, t\in[0,T]   $ and coefficients of
the integral equations  are locally Lipschitz continuous if $y\neq 0.$

Let $\omega$ be such that $ Y_\ve(\tau_\ve(\delta))=\delta$.
Then
\[
\ba
|X_\ve(\tau_\ve(\delta)+t)-X^{+,X_\ve(\tau_\ve(\delta)),Y_\ve(\tau_\ve(\delta))}(t)|
+|Y_\ve(\tau_\ve(\delta)+t)-Y^{+,X_\ve(\tau_\ve(\delta)),Y_\ve(\tau_\ve(\delta))}(t)|=\\
\left(|X^{+, x,\delta}_{(f)}(t)- X^{+, x,\delta}(t)| +
 |Y^{+, x,\delta}_{(f)}(t)- Y^{+, x,\delta}(t)|\right) \Big|_{x=X_\ve(\tau_\ve(\delta))},
\ea
\]
where
\[
f(t):= \left( \eps \int_{\tau_\ve(\delta)}^{\tau_\ve(\delta)+t} b\big(X_\eps(s), Y_\eps(s)\big) \d B(s),
\ \eps \int_0^t \beta\big(X_\eps(s), Y_\eps(s)\big) \d W(s)\right),
\]

Since $b$ and $\beta$ are bounded we have the uniform convergence in probability:
\[
\eps  \sup_{t\in{[0,T]}}\left(
 | \int_{\tau_\ve(\delta)}^{\tau_\ve(\delta)+t} b\big(X_\eps(s), Y_\eps(s)\big) \d B(s)|+
|\int_0^t \beta\big(X_\eps(s), Y_\eps(s)\big) \d W(s)|\right) \overset{\P}\to 0, \ \ve\to0
\]
for any $\delta>0.$

This, \eqref{eq:estim_dist}, \eqref{eq:1072}, and Lemma \ref{lem:int_etim1} give us convergence
\[
\sup_{t\in{[0,T]}}\left(
|X_\ve(\tau_\ve(\delta)+t)-X^{X_\ve(\tau_\ve(\delta)),Y_\ve(\tau_\ve(\delta))}(t)|
+|Y_\ve(\tau_\ve(\delta)+t)-Y^{X_\ve(\tau_\ve(\delta)),Y_\ve(\tau_\ve(\delta))}(t)|\right) \overset{\P}\to 0, \ \ve\to0
\]
for any $\delta>0.$


{\bf Step 5.}
The proof of the Theorem follows from Step 4 and the next estimate of probabilities $\P({Y_\ve(\tau_\ve(\delta))=\pm\delta})$.
\begin{lem}\label{lem:probab_exit}
$\forall \mu>0, \ \delta_0>0\ \  \exists \delta\in(0,\delta_0)$:
\[
p^+-\mu\leq\liminf_{\ve\to0}\P({Y_\ve(\tau_\ve(\delta))=\delta}) \leq \limsup_{\ve\to0} \P({Y_\ve(\tau_\ve(\delta))=\delta})\leq p^+ +\mu,
\]
\[
p^--\mu\leq\liminf_{\ve\to0}\P({Y_\ve(\tau_\ve(\delta))=-\delta}) \leq \limsup_{\ve\to0} \P({Y_\ve(\tau_\ve(\delta))=-\delta})\leq p^- +\mu,
\]
where $p^\pm$ are defined in  \eqref{eq_p+-}.
\end{lem}
\begin{proof}[Proof of Lemma \ref{lem:probab_exit}]
Let $\nu>0$ be arbitrary. Select $\delta_1>0$ such that
\bel{eq:930}
|\vf^\pm(x,y)- \vf^\pm(x^0,0)|\leq \nu,\ \ 0< (\beta^\pm(x^0,0))^2-\nu<
 (\beta^\pm(x,y))^2<(\beta^\pm(x^0,0))^2+\nu
\ee
 as $|x-x^0|<\delta_1, \ |y|\in [0, \delta_1].$

Set $\sigma_{\ve}(\delta):=\inf\{ t\geq 0\ : \ |X_\ve(t)-x^0|\geq \delta\}.$

It follows from \eqref{eq:estim_dist} that
$\P\big(\sigma_{\ve}(\delta^{\frac{{1-\gamma}}{6 }})<\tau_\ve(\delta)\big)< C
 \delta^{\frac{{1-\gamma}}{6 }}$ for small $\ve$. Hence, if $\delta^{\frac{{1-\gamma}}{6 }}<\delta_1,$ then  with
probability greater than $1-C \delta^{\frac{{1-\gamma}}{6 }}$ the process $Y_\ve$ exits
$[-\delta,\delta]$ before $X_\ve$ exits $[-\delta_1,\delta_1]$. Hence, without loss of generality we will assume that \eqref{eq:930} is satisfied for all $(x,y)$.

Set
\[
s_\ve(y):=
\begin{cases}
\int_0^y \exp\{-\frac{2 (\vf^+(x^0,0)+\nu) z^{\gamma+1}}{\ve^2(\gamma+1)(\beta^+(x^0,0)^2-\nu)}\} dz, \ y\geq 0;\\
\\
\int_0^y \exp\{-\frac{2 (\vf^-(x^0,0)-\nu) |z|^{\gamma+1}}{\ve^2(\gamma+1)(\beta^-(x^0,0)^2+\nu)}\} dz, \ y\leq 0.
\end{cases}
\]
Then
\[
\vf(x,y)y^{\gamma} s_\ve'(y)+ \frac{\ve^2}2 \beta^2(x,y) s''_\ve(y)\leq 0
\]
for all $x,y$ (recall that we assume that \eqref{eq:930} is satisfied for all $(x,y)$
).

So
\[
0\geq\E s_\ve(Y_\ve(\tau_\ve(\delta)))=
s_\ve(\delta)\P(Y_\ve(\tau_\ve(\delta))=\delta)
+ s_\ve(-\delta)\P(Y_\ve(\tau_\ve(\delta))=-\delta)=
\]
\[
s_\ve(\delta)\P(Y_\ve(\tau_\ve(\delta))=\delta)
+ s_\ve(-\delta)(1-\P(Y_\ve(\tau_\ve(\delta))=\delta))=
(s_\ve(\delta)-s_\ve(-\delta))\P(Y_\ve(\tau_\ve(\delta))=\delta)
+ s_\ve(-\delta).
\]
Therefore
\[
\limsup_{\ve\to0} \P({Y_\ve(\tau_\ve(\delta))=\delta})\leq
\lim_{\ve\to0} \frac{-s_\ve(-\delta)}{s_\ve(\delta)-s_\ve(-\delta)}=
\]
\[
\lim_{\ve\to0} \frac{\int_{-\delta}^0 \exp\{-\frac{2 (\vf^-(x^0,0)-\nu) |z|^{\gamma+1}}{\ve^2(\gamma+1)((\beta^-(x^0,0))^2+\nu)} \} dz}
{ \int_{-\delta}^0 \exp\{-\frac{2 (\vf^-(x^0,0)-\nu) |z|^{\gamma+1}}{\ve^2(\gamma+1)((\beta^-(x^0,0))^2+\nu)} \} dz +\int_{0}^\delta \exp\{-\frac{2 (\vf^+(x^0,0)+\nu) z^{\gamma+1}}{\ve^2(\gamma+1)((\beta^+(x^0,0))^2-\nu)}\} dz}=
\]
\[
\frac{\left(\frac{  \vf^+(x^0,0)+\nu  }{ (\beta^+(x^0,0))^2-\nu }\right)^{\frac{1}{\gamma+1}}}
{\left(\frac{  \vf^-(x^0,0)-\nu  }{ (\beta^-(x^0,0))^2+\nu }\right)^{\frac{1}{\gamma+1}}+ \left(\frac{  \vf^+(x^0,0)+\nu  }{ (\beta^+(x^0,0))^2-\nu }\right)^{\frac{1}{\gamma+1}}}.
\]
Here we used the following
\[
\int_0^\delta \exp\{ -\frac{Az^{\gamma+1}}{\ve^2} \}dz=
\frac{1}{1+\gamma}  (\frac{\ve^2}{A})^{\frac{1}{1+\gamma} }\int_0^{\frac{A\delta^{\gamma+1}}{\ve^2}}    e^{ -t} t^{\frac{1}{1+\gamma}-1} dt \sim \frac{1}{1+\gamma}  (\frac{\ve^2}{A})^{\frac{1}{1+\gamma} }\Gamma(\frac{1}{1+\gamma} ),\ \ve\to0
\]
for any $A>0, \delta>0.$

Since, $\nu$ was arbitrary, this completes the proof of Lemma \ref{lem:probab_exit} and Theorem \ref{thm:main_small_noise_rep}.
\end{proof}

\section{Proof of Theorem \ref{thm:main_small_noise_attr}}\label{Th2.2}

At the beginning notice that
\be\label{eq:Yto0}
Y_\ve\Rightarrow 0,\quad \ve\to0.
\ee
Indeed, by It\^o's formula we have
\[
 Y_\ve^2(t)\leq  C \ve^2 + 2\ve \int_0^t Y_\ve(s)\beta\big( X_\eps(s),Y_\eps(s)\big)\,  \d W(s),
\]
where $C$ is independent of $\ve.$ Hence we get an estimate
\[
\forall {\ve>0}\ \ \  \sup_{t\in[0,T]} \E Y_\ve^2(t)   \leq C\eps^2.
\]
It follows from the Doob inequality that
\[
\sup_{t\in[0,T]} |2\ve \int_0^t Y_\ve(s)\beta\big( X_\eps(s),Y_\eps(s)\big)\,  \d W(s)  | \overset{\P}\to  0, \quad \ve\to0.
\]
This completes the proof of   \eqref{eq:Yto0}.

Let $\delta>0$ be a fixed number.
Notice that
 \[
\eps^{-\delta}Y_\eps(t)=  \eps^{ \delta(\gamma-1)} \int_0^t\phi (X_\eps(s), Y_\eps(s))
(\eps^{-\delta}Y_\eps(t))^\gamma \, \d t+
\]
\[
\eps^{1-\delta-\frac{\delta(\gamma-1)}{2}}\, \eps^{ \frac{\delta(\gamma-1)}{2}} \int_0^t\beta\big( X_\eps(s),Y_\eps(s)\big)\,  \d W(s)=
\]
\[
 \int_0^t\phi (X_\eps(s), \eps^{\delta}\eps^{-\delta}Y_\eps(s))
(\eps^{-\delta}Y_\eps(t))^\gamma \, \d ( \eps^{ \delta(\gamma-1)}  t)+
\]
\[
\eps^{1-\frac{\delta(\gamma+1)}{2}} \int_0^t\beta\big( X_\eps(s), \eps^{\delta}\eps^{-\delta} Y_\eps(s)\big)\,   \d W_\eps(\eps^{ \delta(\gamma-1)}s),
\]
where $W_\eps(t)=\eps^{ \frac{\delta(\gamma-1)}{2}} W (\eps^{- \delta(\gamma-1)} t)$ is a Wiener process.

If $1-\frac{\delta(\gamma+1)}{2}=0,$ i.e., $\delta=\frac{2}{\gamma+1},$ then the process
$\tilde Y_\eps(t):=\eps^{-\delta}Y_\eps(t)=\eps^{ \frac{-2}{\gamma+1}}Y_\eps(t)$ satisfies the SDE
\[
\tilde Y_\eps(t)=
  \int_0^t\phi (X_\eps(s), \eps^{\frac{2}{\gamma+1}}\tilde Y_\eps(s))
 \tilde Y^\gamma_\eps(s)  \, \d (\eps^{  \frac{2(\gamma-1)}{\gamma+1} } s)+
  \int_0^t\beta\big( X_\eps(s), \eps^{\frac{2}{\gamma+1}}\tilde Y_\eps(s) \big)\,   \d W_\eps(\eps^{  \frac{2(\gamma-1)}{\gamma+1} }s).
\]
Set $\tilde \eps=\eps^{  \frac{2(1-\gamma)}{\gamma+1} }.$ Therefore
\be\label{eq:small_trasformed}
\ba
\d X_{  \eps}(t)&=a_{\tilde \eps}\big(X_{  \eps}(t), \tilde Y_{  \eps}(t)\big)\, \d t
+b_{\tilde \eps}\big(X_{  \eps}(t), \tilde Y_{  \eps}(t)\big)\, \d B(t),
\\
\d \tilde Y_{  \eps}(t)&=  \alpha_{\tilde \eps} (X_{\eps}(t), \tilde Y_{\eps}(t)) 
\, \d \tilde\eps^{-1}t+
   \beta_{\tilde \eps}\big( X_{\eps}(t),\tilde Y_{\eps}(t)\big)\, \d W_{\tilde \eps}({\tilde \eps}^{-1}t),
 \ea
\ee
where
\[
a_{\tilde \eps}(x,y)=\psi(x,  \eps^{\frac{2}{\gamma+1}}y)= \psi(x, \tilde\eps^{\frac{1}{1-\gamma}}y),
  \ b_{\tilde \eps}(x,y)=
{\eps} b(x,\eps^{\frac{2}{\gamma+1}}y)
={\tilde \eps}^{(\gamma+1)/2(1-\gamma)}b(x,\tilde\eps^{\frac{1}{1-\gamma}}y),
\]
\[
 \alpha_{\tilde \eps}(x,y)=\phi (x, \eps^{\frac{2}{\gamma+1}}y)   y^\gamma
= \phi (x, \tilde\eps^{\frac{1}{1-\gamma}}y)   y^\gamma,\
 \beta_{\tilde \eps}(x,y)=\beta(x, \eps^{\frac{2}{\gamma+1}}y)=\beta(x, \tilde\eps^{\frac{1}{1-\gamma}}y).
\]
 We see that   the system \eqref{eq:small_trasformed} has the form \eqref{DA_fully coupled}.
Let us apply Theorem \ref{thm:mainAP}, where $k=1$,
\[
\begin{aligned}
a^\ve(x,y):= a_{\tilde \ve}(x,y)
, \
\sigma^\ve(x,y)=c^\ve(x,y,u)= C^\ve(x,y,z)=0,\\
A^\ve(x,y):=\alpha_{\tilde \eps}(x,y)
,\
\Sigma^\ve(x,y):=   \beta_{\tilde \eps}(x,y)
,\\
\xi_\ve(t):={\tilde \eps}^{(\gamma+1)/2(1-\gamma)} \int_0^t
b(X_\ve(s), \tilde\eps^{\frac{1}{1-\gamma}} \tilde Y_{\tilde \ve}(s)) dB(s).
\end{aligned}
\]

Conditions
 $\mathbf{H}_0$, $\mathbf{H}_1$, and $\mathbf{H}_2$  are obviously true.

Conditions $\mathbf{H}_3$, $\mathbf{H}_4$ are
 satisfied with
$$
A(x,y)= (\phi^+(x, 0)\1_{y>0} + \phi^-(x, 0)\1_{y<0})  y^\gamma ,\  \ \Sigma(x,y)=
\beta^+(x,0)\1_{y\geq 0}+ \beta^-(x,0)\1_{y< 0},
$$
$$
 C(x,y,z)=0,\ \  \nu=\mu=0.
$$

Without loss of generality we will assume that
\bel{eq:1209}
\vf^\pm(x,0)\leq c <0   \ \mbox{ for all}\   x\in\mbR^d,
\ee
where $c  $ is a  constant.
The general case can be considered using a localization.
Hence, condition $\mathbf{H}_5$ is satisfied with $\kappa=\gamma$.

Consider equation with frozen coefficients
\begin{equation}\label{SDEfrozen}
\begin{aligned}
\d y^{(x)}(t) &=     (\phi^+(x, 0)\1_{y^{(x)}(t)>0} + \phi^-(x, 0)\1_{y^{(x)}(t)<0})(y^{(x)}(t))^\gamma  \, \d  t
\\&+
  \left(
 \beta^+(x,0)\1_{y^{(x)}(t)\geq 0}+ \beta^-(x,0)\1_{{y^{(x)}(t)<0}}
\right)
 \, \d W(t).
 \end{aligned}
\end{equation}

{
Existence and uniqueness of a weak solution to   equation with frozen
coefficients, and the strong Markov property
follows from \cite{EngelbertSchmidt3}.
Hence condition $\mathbf{F}_0$ holds true.
}

{
 To verify condition $\mathbf{F}_1$, we modify the argument from \cite[Section~3.3.2]{K_book}. Because the diffusion coefficient in  \eqref{SDEfrozen}  is discontinuous, we do not have a good reference to state that the transition probability density $p^{(x)}(t,y,y')$ is continuous in $x,y,y'$. In order to overcome this minor difficulty we use the following localization argument. Consider the SDE
\begin{equation}\label{SDEfrozen+}
\begin{aligned}
\d y^{(x,+)}(t) &=     \phi^+(x, 0)(|y^{(x,+)}(t)|\wedge 2)^\gamma\sgn( y^{(x,+)}(t)) \, \d  t
+\beta^+(x,0)
 \, \d W(t).
 \end{aligned}
\end{equation}
This is an SDE with a constant diffusion coefficient and bounded and H\"older continuous drift coefficient, hence the standard analytic theory (e.g. \cite{Fr64}) yields that its transition probability density $p^{(x,+)}(t,y,y')$ is continuous in $x,y,y'$. Then for $y_0=1$ and every $t_0>0$ it holds that
$$
\sup_{|x|\leq R}\|P_{t_0}^{(x,+)}(y, \d y')- P_{t_0}^{(x,+)}(y_0, \d y')\|_{TV}=
\sup_{|x|\leq R}\int| p_{t_0}^{(x,+)}(y,  y')- p_{t_0}^{(x,+)}(y_0,  y')|\, \d y'\to 0, \quad y\to y_0. 
$$
The coefficients of the equations \eqref{SDEfrozen}, \eqref{SDEfrozen+} coincide on $[0,2]$, and thus the laws of the solutions to these equations, stopped at the moment of exit from $[0,2]$, coincide. Taking $t_0$ small enough, we can guarantee that each of these solutions stay in $[0,2]$ up to the time $t_0$ with probability $\geq \frac56$ if the initial value $y$ stays in $[\frac12, \frac32]$. By the coupling characterization of the TV distance (the `Coupling Lemma', e.g. \cite[Theorem 2.2.2]{K_book}), this yields that, for such $t_0$,
$$
\sup_{y_1,y_2\in [\frac12, \frac32], |x|\leq R}\|P_{t_0}^{(x)}(y_1, \d y')- P_{t_0}^{(x,+)}(y_2, \d y')\|_{TV}\leq \left(1-\frac56\right)+\left(1-\frac56\right)=\frac13.
$$
Combining these two estimates we see that there exist $t_0>0$ and $r>0$ small enough, so that
$$
\sup_{y_1,y_2\in[1-r,1+r], |x|\leq R}\|P_{t_0}^{(x)}(y_1, \d y')- P_{t_0}^{(x)}(y_2, \d y')\|_{TV}<\frac34;
$$
in the RHS we could actually take any number $>\frac13+\frac13=\frac23$. This proves the local Dobrushin condition in a small ball centered at $y_0=1$. To extend this condition to a large ball $|y|\leq R$, we use another standard argument, based on the support theorem. Namely, $y^{(x)}$ can be represented as an image of a Brownian motion under the time change and the change of measure; see \cite{IkedaWatanabe}. Since the Wiener measure in $C_0(0,\infty)$ has a full topological support, it is easy to show using this representation that, for any $t_1>0$, there exists $\delta>0$ such that 
$$
P_{t_1}^{(x)}(y, [1-r,1+r])\geq \delta, \quad |x|\leq R, \quad |y|\leq R.
$$
Take $h=t_0+t_1$ and for $x,y_1,y_2$ with $|x|\leq R, |y_1|\leq R, |y_2|\leq R$ consider two processes $Y_t^1, Y_t^2$ which start at $y_1,y_2$ respectively, solve \eqref{SDEfrozen} independently up to the time $t_1$, and then provide the maximal coupling probability on the time interval $[t_1,t_1+t_0]$, conditioned on their values at the time $t_1$ (we can construct such a process using the Coupling Lemma for probability kernels, \cite[Theorem~2.2.4]{K_book}.) Then 
$$
\begin{aligned}
\|P_{h}^{(x)}(y_1, \d y')&- P_{h}^{(x)}(y_2, \d y')\|_{TV}\leq \P(Y_h^1\not=Y_h^2)
\\&=\int_{\mathbf{R}^2}\|P_{t_0}^{(x)}(z_1, \d y')- P_{t_0}^{(x)}(z_2, \d y')\|_{TV}P_{h}^{(x)}(y_1, \d z_1)P_{h}^{(x)}(y_2, \d z_2)
\\&\leq 1-P_{t_1}^{(x)}(y_1, [1-r,1+r])P_{t_1}^{(x)}(y_2, [1-r,1+r])
\\&+\int_{[1-r,1+r]^2}\|P_{t_0}^{(x)}(z_1, \d y')- P_{t_0}^{(x)}(z_2, \d y')\|_{TV}P_{h}^{(x)}(y_1, \d z_1)P_{h}^{(x)}(y_2, \d z_2)
\\&\leq 1+\left(-1+\frac34\right)P_{t_1}^{(x)}(y_1, [1-r,1+r])P_{t_1}^{(x)}(y_2, [1-r,1+r])
\\&\leq 1-\frac{\delta^2}{4}.
\end{aligned}$$
for any $|x|\leq R, |y_1|\leq R, |y_2|\leq R$, which completes the proof of $\mathbf{F}_1$. 
}

The invariant probability measure $\pi^{(x)}(\d y)$  equals, see \cite[Exercise 5.40]{Karatzas}:
\[
\pi^{(x)}(dy)= 
c(x)  \left(\exp\left\{  \frac{\phi^+(x, 0)}{  (\beta^+(x,0))^2} \frac{y^{\gamma+1}}{(\gamma+1) } \right\} \1_{y\geq 0}+
\exp\left\{  \frac{\phi^-(x, 0)}{  (\beta^-(x,0))^2}
\frac{y^{\gamma+1}}{(\gamma+1) }  \right\}\1_{y<0} \right) dy,
\]
where
\[c(x)^{-1}=
\int_\bR \left(\exp\left\{  \frac{\phi^+(x, 0)}{  (\beta^+(x,0))^2}  \frac{y^{\gamma+1}}{(\gamma+1) } \right\} \1_{y\geq 0}+
\exp\left\{  \frac{\phi^-(x, 0)}{  (\beta^-(x,0))^2}
\frac{y^{\gamma+1}}{(\gamma+1) }  \right\}\1_{y<0} \right) dy
=
\]
\[
 \frac{\Gamma(\frac{1}{\gamma+1}) }{(\gamma+1) }
\left(\left(\frac{(\gamma+1)(\beta^+(x,0))^2}{\phi^+(x,0)}\right)^{ \frac{1}{\gamma+1}}+
\left(\frac{(\gamma+1)(\beta^-(x,0))^2}{\phi^-(x,0)}\right)^{ \frac{1}{\gamma+1}} \right).
\]
 Condition $\mathbf{H}_6$ is
 satisfied with
 $
 a(x,y) = \psi_+(x, 0)\1_{y>0}+ \psi_-(x, 0)\1_{y<0},  \ \sigma(x,y)=c(x,y,z)=0,
 $
 and  $B=\mbR^d\times\{0\}.$

The averaged coefficient
\[
\overline{a}(x)=\psi^+(x,0) \pi^{(x)}([0,\infty))+ \psi^-(x,0) \pi^{(x)}((-\infty,0))=
\]
\[
\psi^+(x,0)
\frac{\left( \frac{(\beta^+(x,0))^2}{\phi^+(x,0)} \right)^{ \frac{1}{\gamma+1}}}
{\left( \frac{(\beta^-(x,0))^2}{\phi^-(x,0)} \right)^{ \frac{1}{\gamma+1}}+
\left( \frac{(\beta^+(x,0))^2}{\phi^+(x,0)} \right)^{ \frac{1}{\gamma+1}}} +
 \psi^-(x,0)
\frac{\left( \frac{(\beta^-(x,0))^2}{\phi^-(x,0)} \right)^{ \frac{1}{\gamma+1}}}
{\left( \frac{(\beta^-(x,0))^2}{\phi^-(x,0)} \right)^{ \frac{1}{\gamma+1}}+
\left( \frac{(\beta^+(x,0))^2}{\phi^+(x,0)} \right)^{ \frac{1}{\gamma+1}}}
\]
  is Lipschitz continuous, $ \bar b(x)=0$, $ \overline{K}_{(\rho)}(x, \d v)= \overline{K}^{(\rho)}(x , \d v)=0$.
 So, condition
$\mathbf{A}_0 $ holds true and
 the corresponding martingale problem
has a unique solution.

 This with \eqref{eq:Yto0} concludes the proof.

\section{Proof of Theorem \ref{thm:mainAP}} \label{section:proof_AP}
The weak compactness of the family $\{X_\eps, \eps>0\}$ in $\mathbb{D}([0, \infty), \mathbb{R}^d)$ follows, in a standard way, from the negligibility  assumption $\mathbf{H}_0$ and the boundedness assumptions $\mathbf{H}_1,\mathbf{H}_2$. Under the assumptions of the theorem, for any $C^\infty_0$-function $\phi$ the function $L\phi$ is continuous and bounded. Hence, in order to prove that any weak limit point of the family $\{X_\eps, \eps>0\}$ as $\ve\to0$ solves the MP \eqref{MP}, it is enough to show that,  for any $C^\infty_0$-function $\phi$, any
$s_1\dots,s_q<s<t$, and any continuous and bounded function $\Phi:\Re^{d\times q}\to \Re$
\be\label{eq1}
\E^\eps \Phi(X_{\eps}(s_1), \dots, X_\eps(s_q))\left[\phi(X_\eps(t))-\phi(X_\eps(s))-\int_s^tL\phi(X_\eps(r))\, dr\right]\to 0, \quad \eps\to 0,
\ee
we denote by $\E^\eps$ the expectation w.r.t. $\P^\eps$. Denote
\be\label{eq_tilde_X}
\ba
\wt X_\eps(t)&=X_\eps(t)-\xi_\eps(t)
\\&=X_\eps(0)+\int_0^t a^\eps\big(X_\eps(s), Y_\eps(t)\big)\, \d s+
\int_0^t\sigma^\eps\big(X_\eps(s), Y_\eps(s)\big)\, \d B^\eps_s
\\&+ \int_0^t\int_{\mbR^m} c^\eps\big(X_\eps(s-), Y_\eps(s-),u\big)\Big[N^\eps(\d u,\d s)-1_{|u|\leq \rho}\nu^\eps(\d u)ds\Big].
\ea
\ee
Observe that $L\phi$ is a bounded and continuous function. So, by $\mathbf{H}_0$ relation \eqref{eq1} is equivalent to
\be\label{eq1bis}
\E^\eps \Phi(\wt X_{\eps}(s_1), \dots, \wt X_\eps(s_q))\left[\phi(\wt X_\eps(t))-\phi(\wt X_\eps(s))
-\int_s^t L\phi(\wt X_\eps(r))\, dr\right]\to 0, \quad \eps\to 0.
\ee
Denote
$$
b^\eps(x,y)=\sigma^\eps(x,y)(\sigma^\eps(x,y))^*, \quad K_{(\rho)}^\eps(x,y, A)
=\nu^\eps(\{u: |u|\leq \rho,  c^\eps(x,y,u)\in A\}),
$$
$$K^{(\rho),\eps}(x,y, A)=\nu^\eps(\{u: |u|> \rho,  c^\eps(x,y,u)\in A\}),
$$
and
$$
\ba
\mathcal{L}^\eps \phi(x,y)&=\nabla \phi(x)\cdot a^\eps(x,y)+\frac{1}{2}\nabla^2 \phi(x)\cdot b^\eps(x,y)
+\int_{\mbR^m}\Big( \phi(x+v)-\phi(x)-\nabla \phi(x)\cdot v\Big){K}_{(\rho)}^\eps(x, \d v)\\&+\int_{\mbR^m}\Big( \phi(x+v)-\phi(x)\Big){K}^{(\rho),\eps}(x,y, \d v).
\ea
$$
Then by the It\^o formula we have
\be\label{eq:res+mart}
\phi(\wt X_\eps(t))-\phi(\wt X_\eps(s))=\int_s^t\mathcal{L}^\eps\phi(X_\eps(r),Y_\eps(r))\, dr+ (martingale\ part).
\ee
Applying $\mathbf{H}_0$ once again, we get that, to prove \eqref{eq1} and \eqref{eq1bis}, it is enough to prove, for  any
$s_1\dots,s_q<t$,
\be\label{eq1prim}
\E^\eps \Phi( X_{\eps}(s_1), \dots,  X_\eps(s_q))\Big(\mathcal{L}^\eps\phi(X_\eps(t),Y_\eps(t))-L\phi(X_\eps(t))\Big)\to 0, \quad \eps\to 0.
\ee

Before proving \eqref{eq1prim}, we formulate and prove two auxiliary statements.

\subsection{Auxiliaries, I: uniform ergodic rate for the frozen microscopic dynamics}

\begin{prop}\label{p1}
  Let conditions $\mathbf{H}_1 - \mathbf{H}_5,$ $\mathbf{F}_0,\mathbf{F}_1$ hold. {
If $\kappa\in(0,1)$ and $p>0$ are from these conditions}, then for every $R>0$  there exists $C$ such that for any $x,y$ with $|x|\leq R, |y|\leq R$
 \be\label{convTV}
  \|P_t^{(x)}(y,\d y')-\pi^{(x)}(\d y')\|_{TV}\leq Ct^{-\frac{p+\kappa-1}{1-\kappa}}.
  \ee
  If $\kappa\geq 1$, then there exists $a>0$ such that,  for every $R>0$ and  any $x,y$ with $|x|\leq R, |y|\leq R$,
   $$
  \|P_t^{(x)}(y,\d y')-\pi^{(x)}(\d y')\|_{TV}\leq Ce^{-at}
  $$
  with a constant $C$ depending on $R$.
\end{prop}
\begin{proof}
  The required statement is actually obtained, though not in this precise form, in \cite[Section~3]{K_book}. The difference between the current situation and the one studied in \cite{K_book} is that  the ergodic rates were obtained there for individual processes (while here we have a family indexed by $x$) and separately for diffusions and L\'evy driven SDEs (while here we have both types of the noise involved simultaneously). This difference is not crucial, and we just give a short outline of the argument, referring to \cite{K_book} for details.

  The convergence conditions $\mathbf{H}_3,\mathbf{H}_4$ yield that the bounds from the conditions $\mathbf{H}_1,\mathbf{H}_2$ and the drift condition $\mathbf{H}_5$ remain true for the limiting coefficients $A(x,y), \Sigma(x,y), C(x,y,z)$ and L\'evy measure $\mu(\d u)$. Then we have  the following: if $V\in C^2$ is a function such that $V(y)\geq 1$ and $V(y)=|y|^p, |y|\geq 2$, then for any $x\in \mbR^d$ the
 semimartingale decomposition holds
\be\label{semimart}
 V(y^{(x)}(t))=V(y^{(x)}(0))+\int_{0}^{t}\mathcal{A}V(x,y^{(x)}(s))\, \d s+(martingale\ part),
 \ee
 where the function $\mathcal{A}V(x,y)$ satisfies
\be\label{semimart_bound}
 \mathcal{A}V(x,y)\leq \left\{
                         \begin{array}{ll}
                            C_V-a_V V(y)^{\frac{p+\kappa-1}{p}}, & \kappa\in (0,1); \\
                           C_V-a_V V(y), & \kappa\geq 1
                         \end{array}
                       \right.
 \ee
with some constants $C_V, a_V>0$. For the proof of this statement, see \cite{KulikPavlyukevich}, Proposition 2.5.

Given \eqref{semimart}, \eqref{semimart_bound} we can proceed analogously to \cite[Sections~3.3,3.4]{K_book}. Namely, for $\kappa\in (0,1)$ we use
\cite[Theorem~3.2.3]{K_book} and  \cite[Example~3.2.6]{K_book} to show that
\be\label{skeleton}
\E_y \wt V(y^{(x)}({h}))-\wt V(y)\leq  \wt C_V-\wt c_V \wt V(y)^{\frac{p+\kappa-1}{p}},
\ee
where $h$ is the same as in the assumption $\mathbf{F}_1$, $\wt C_V$, $\wt c_V>0$ are some new constants, and $\wt V$ is a new function which is equivalent to $V$ in the sense that, for some positive constants $c_1,c_2$
$$
c_1V\leq \wt V\leq c_2 V.
$$
Following the  proof  of \cite[Theorem~3.2.3]{K_book} and calculations of \cite[Example~3.2.6]{K_book} line by line, we easily see that, because the constants $C_V, c_V$ in \eqref{semimart_bound} do not depend on $x$, the constants $\wt C_V$, $\wt c_V$, $c_1,c_2$ and the function $\wt V$ can be chosen uniformly for $x\in \mbR^d$.

Inequality \eqref{skeleton} is actually the  \emph{Lyapunov condition} for
 the \emph{skeleton chain} $y^{(x,h)}_k=y^{(x)}({kh}), k\geq 0$ for the process $y^{(x)}$,
see \cite[Section~2.8]{K_book}. Combined with the local Dobrushin condition assumed in
 $\mathbf{F}_1$, we get by \cite[Corollary~2.8.10]{K_book} {for $\kappa\in(0,1)$ the inequality}
$$
 \|P_{kh}^{(x)}(y,\d y')-\pi^{(x)}(\d y')\|_{TV}\leq C(1+k)^{-\frac{p+\kappa-1}{1-\kappa}}\wt V(y),
$$
where we have used the identity
$$
\frac{p+\kappa-1}{p}\left(1-\frac{p+\kappa-1}{p}\right)^{-1}=\frac{p+\kappa-1}{1-\kappa}.
$$
Since Lyapunov condition and the local Dobrushin condition are uniform in $x$, the constant $C$ here
 can be chosen uniformly for $x\in \mbR^d$; one can easily check this following line by line
 the proofs of  \cite[Corollary~2.8.10]{K_book} and the theorems it is
based on: \cite[Theorem~2.7.5]{K_book} and  \cite[Theorem 2.8.6]{K_book}.
 Since the total variation distance $\|P_{t}^{(x)}(y,\d y')-\pi^{(x)}(\d y')\|_{TV}$ is non-increasing
in $t$ and $V(y)$ is locally bounded, this completes the proof of the required statement in the case
$\kappa\in (0,1)$.

For $\kappa\geq 1$, we can argue in a completely analogous way, using
 \cite[Corollary~2.8.3]{K_book}.
  \end{proof}

\subsection{Auxiliaries, II: weak convergence of the microscopic dynamics to the frozen one}
Consider the following \emph{microscopic} analogue of \eqref{DA_fully coupled}.
 {Assume that $(x_\eps, y_\eps)$ is a solution (maybe non-unique) to the equation}
\be\label{DA_fully coupled_micro}
\ba
x_\eps(t)&=x_\eps(0)+\eps\int_0^t a^\eps\big(x_\eps(s), y_\eps(t)\big)\, \d s+
\eps^{1/2}\int_0^t\sigma^\eps\big(x_\eps(s), y_\eps(s)\big)\, \d b^\eps_s
\\&+ \int_0^t\int_{\mbR^m} C^\eps\big(x_\eps(s-), y_\eps(s-),u\big)\Big[n^\eps(\d u,\d s)-1_{|u|\leq \rho}\eps \nu^\eps(\d u)ds\Big]+\zeta^\eps(t),
\\
y_\eps(t)&=y_\eps(0)+\int_0^tA^\eps\big(x_\eps(s), y_\eps(s)\big)\, \d s+\int_0^t\Sigma^\eps\big( x_\eps(s),y_\eps(s)\big)\, \d w_{s}^\eps
\\&+\int_{0}^{t} \int_{\mbR^l} c^\eps\big(x_\eps(s-), y_\eps(s-),z\big)\Big[q^\eps(\d z,\d s)-1_{|z|\leq \rho}\mu^\eps(\d z)\d s\Big],
\ea
\ee
where $b^\eps_t, w^\eps_t$ are Brownian motions and $n^\eps(\d u,\d t), q^\eps(\d z,\d t)$ are
 Poisson point measures on a common filtered probability space
$(\wt \Omega^\eps, \wt \cF^\eps, \wt \P^\eps)$, and the random measures
$n^\eps(\d u,\d t)$, $q^\eps(\d z,\d t)$ have the intensity  measures $\eps \nu^\eps(\d u)\d t$
 and $\mu^\eps(\d z)\d t$, respectively,   {$\zeta^\eps(t)$ is an adapted c\`adl\`ag process}.

System \eqref{DA_fully coupled_micro} naturally appears e.g. if we consider the original system \eqref{DA_fully coupled} at the `microscopic time scale' $\eps t$ with an initial time shift by $t_0$:
\be\label{micro}
x_\eps(t)=X_\eps(t_0+\eps t), \quad y_\eps(t)=Y_\eps(t_0+\eps t), \quad \mbox{and} \quad
 \zeta^\eps(t)=\xi_\eps(t_0+\eps t)-\xi_\eps(t_0).
\ee

For a fixed pair of functions {$(\rho(\eps), \varrho(\eps)) $ such that $\rho(\eps)\to 0 $ and $ \varrho(\eps)\to 0 $ as $\eps\to0,$}
and constants $R>0, T>0$ denote by $\mathcal{K}(\rho, \varrho, R,T)$ the class of all families
$\{(x_\eps, y_\eps), \eps>0\}$ which satisfy  \eqref{DA_fully coupled_micro}
 on \emph{some} probability space with {non-random initial values $x_\eps(0), y_\eps(0),$}  $|x_\eps(0)|\leq R, |y_\eps(0)|\leq R$ and
$$
\wt \P^\eps \left(\sup_{s\leq T}|\zeta_\eps(s)|>\rho(\eps)\right)\leq \varrho(\eps).
$$

\begin{prop}\label{p32} Let conditions $\mathbf{H}_1 - \mathbf{H}_5,$ $\mathbf{F}_0$ hold.
 Then for  any $0<t<T$ and any bounded  continuous function $f:\mbR^d\times\mbR^k\to \mbR$ and $R>0$,
\be\label{bound}
\sup_{\{(x_\eps,y_\eps)\}\in \mathcal{K}(\rho, \varrho, R,T)}\Big|\wt \E^\eps f(x_\eps(t),y_\eps(t))-P_t^{frozen}f(x,y)\Big|_{x=x_\eps(0), y=y_{\eps}(0)}\Big|\to 0, \quad \eps\to 0,
\ee
{  where
$$
P_t^{frozen}f(x,y)=\int_{\mbR^k} f(y')P_t^{(x)}(y, \d y'), \quad t\geq 0.
$$
}
\end{prop}
\begin{proof} Assuming the contrary, we will have that there exists a sequence $x_{\eps_n}(\cdot),y_{\eps_n}(\cdot)$ of solutions to  \eqref{DA_fully coupled_micro} with $|x_{\eps_n}(0)|\leq R, |y_{\eps_n}(0)|\leq R$ such that
\be\label{asump}
\Big(\wt \E^{\eps_n} f(x_{\eps_n}(t),y_{\eps_n}(t))-P_t^{(x)}f(x,y)\Big|_{x=x_{\eps_n}(0), y=y_{\eps_n}(0)}\Big)\not \to 0, \ \ n\to\infty.
\ee
Without loss of generality, after passing to a subsequence, we can assume that  $x_{\eps_n}(0)\to x_*$ and
$y_{{\eps_n}}(0)\to y_*$ as $n\to\infty$. Then it is easy to show that, for any $c>0$,
\be\label{x_conv}
\lim_{n\to\infty}\wt \P^{\eps_n}\left(\sup_{s\in[0, T]}|x_{\eps_n}(s)-x_*|>c\right)= 0.
\ee
Next, denote by $\P^*$ the law in $\mathbf{D}([0,T], \mbR^k)$ of $y^{(x_*)}$ with $y^{(x_*)}(0)=y_*$.
Since the $\P^*$-probability for $y(\cdot)$ to have a jump at the point $t$ is $0$,
 the function $F(y(\cdot))=f(y(t))$ is a.s. continuous on $\mathbf{D}([0,T], \mbR^k)$. Thus, in order to prove that  \eqref{asump} fails, it is enough to show that the laws of  $y_{\eps_n}, n\geq 1$ weakly converge in $\mathbf{D}([0,T], \mbR^k)$ to $\P^*$. Such a statement is quite standard, and we just outline its proof here.

By \eqref{x_conv},  the continuity assumption $\mathbf{H}_3$,  and convergence of the noise
 \noindent$\mathbf{H}_4$ it is easy to prove that any weak limit point to $\{y_{\eps_n}\}$
solves \eqref{frozen}.
By the weak uniqueness assumption $\mathbf{F}_0$, this yields that
 any weak limit point to $\{y_{\eps_n}\}$ has the law $\P^*$.

That is, to prove the required weak convergence it is enough to prove that $\{y_{\eps_n}\}$ is weakly compact in $\mathbf{D}([0,T], \mbR^k)$.

 To prove the weak compactness, we use $L_2$-moment bounds for the increments of the process $y_{\eps_n}$ combined with a truncation of the large jumps. Namely, by $\mathbf{H}_2$ for any fixed $\delta>0$ there exists $Q_\delta$ such that
$$
\wt \P^\eps\big(N([0,T]\times \{|z|>Q_\delta\})>0\big)<\delta, \quad \eps>0.
$$
Thus it is enough to prove weak compactness for every `truncated' family  $\{y_{\eps_n, Q}\}, Q>0$, where $y_{\eps, Q}$ satisfies an analogue of \eqref{DA_fully coupled_micro} with the integral for $q^\eps$ taken over $\{|z|\leq Q\}$ instead of $\mbR^l$. For such a `truncated' family,
applying \cite[Proposition 2.5]{KulikPavlyukevich} we get that
\be\label{semimart_2}
 |y_{\eps_n,Q}(s)|^2=|y_{\eps_n,Q}(0)|^2+\int_{0}^{s}H(r)\, \d r+(martingale\ part),
 \ee
where $H$ is bounded.  Combining this with the maximal  martingale inequality, we get that
$$
\wt \E^{\eps_n}\sup_{s\in [0,T]}|y_{\eps_n,Q}(s)|^2
$$
is bounded. Since the coefficient $A^\eps(x,y)$ is  bounded locally in $y$, the above bound and the (uniform) bounds for $C^\eps$, $\mu^\eps$ from
 $\mathbf{H}_1$, $\mathbf{H}_2$ yield the required weak compactness of  $\{y_{\eps_n}\}$. Summarizing all the above, we have that $\{y_{\eps_n}\}$ weakly converges to $\P^*$. Combined with \eqref{x_conv}, this contradicts to \eqref{asump} and  proves the required statement.
\end{proof}

\subsection{End of the proof of Theorem \ref{thm:mainAP}}

{In this subsection we complete the proof of \eqref{eq1prim}. This will conclude
proof of the Theorem.}

Denote
\be\label{El}
\ba
\mathcal{L} \phi(x,y)
&
=\nabla \phi(x)\cdot a(x,y)+\frac{1}{2}\nabla^2 \phi(x)\cdot b(x,y)
+\int_{\mbR^m}\Big( \phi(x+v)-\phi(x)-\nabla \phi(x)\cdot v\Big)
{K}_{(\rho)}(x, y, \d v)\\
 & +\int_{\mbR^m}\Big( \phi(x+v)-\phi(x)\Big){K}^{(\rho)}(x,y, \d v)=
 \ea
\ee
\[
\ba
  \nabla \phi(x)\cdot a(x,y)+\frac{1}{2}\nabla^2 \phi(x)\cdot b(x,y)
+& \\
 & \int_{\mbR^m} \Big( \phi(x+c(x,y,u))-\phi(x)-\nabla \phi(x)\cdot c(x,y,u) \1_{|u|\leq \rho}\Big)
\nu(du).
\ea
\]
Next, since the set $B$ from the condition  $\mathbf{H}_6$ is open, there exists a sequence of continuous functions $\chi_j(x,y), j\geq 1$ such that
\begin{itemize}
  \item[(i)] $0\leq \chi_j(x,y)\leq 1, j\geq 1$;
  \item[(ii)] each $\chi_j$ has a support compactly embedded to $B$;
  \item[(iii)] for each $x,y$, $\chi_j(x,y)\nearrow\chi_\infty(x,y)=1_B(x,y), j\to \infty$.
\end{itemize}

{Recall the notation $\overline{f}(x)=\int_{\Re^k}f(x,y)\pi^{(x)}(\d y)$.}

The following lemma collects several simple statements used in the proof.

\begin{lem}\label{l1} The following properties hold:
\begin{itemize}
  \item[(a)] { there exists $C>0$ such that}  $|\mathcal{L}^\eps \phi(x,y)|\leq C, \eps>0 $ for all $x,y$;
  \item [(b)] $\mathcal{L}^\eps \phi\to \mathcal{L}\phi, \eps\to 0$ uniformly on each compactum $K\subset B$;
   \item[(c)]  { there exists $C>0$ such that} $|\mathcal{L}\phi(x,y)|\leq C $ for all $ (x,y)\in B$;
\item[(d)] $\overline{\chi_j}(x)\to 1, j\to \infty$ uniformly on $\{|x|\leq R\}$ for each $R>0$;
\item[(e)] $\overline{\chi_j\mathcal{L}\phi}(x)\to L\phi(x), j\to \infty$
uniformly on $\{|x|\leq R\}$ for each $R>0$, { where $L$ is defined in \eqref{MP}};
\item[(f)] for any $T>0$,
$$
\E^\eps|\wt X_\eps(t)-\wt X_\eps(s)|^2\wedge 1\leq C|t-s|,\quad  s,t\in [0,T],
$$
where $\wt X$ is defined in \eqref{eq_tilde_X},

 and
$$
\sup_{t\in [0, T], \eps>0}\P^\eps(|X_\eps(t)|>R)\to 0, \quad R\to \infty;
$$
\item[(g)] for any $T>0$,
$$
\sup_{t\in [0, T], \eps>0}\P^\eps(|Y_\eps(t)|>R)\to 0, \quad R\to \infty.
$$
\end{itemize}

\end{lem}
\begin{proof} Statement (a) follows directly from the assumptions $\mathbf{H}_1$, $\mathbf{H}_2$.   Statement (b) can be derived, in a standard way, using the convergence assumptions $\mathbf{H}_4$, $\mathbf{H}_6$  and the bounds from the  assumptions $\mathbf{H}_1$, $\mathbf{H}_2$.
Statement (c) follows from (a) and (b).

To prove statement (d), we first mention that each function $\overline{\chi_j}$ is continuous by the assumption $\mathbf{F}_0$. These functions converge monotonously, at each $x\in \mbR^d$,  to the function
$$
\overline{\chi_\infty}(x)=\int_{\mbR^k}1_{B}(x,y)\pi^{(x)}(dy)\equiv 1,
$$
where the last identity holds by the assumption $\mathbf{H}_6$. Then the required uniform convergence follow by the Dini theorem.

To prove statement (e), we first use statements (c) and (d) to get
$$
|\overline{\chi_j \mathcal{L}\phi}(x)-\overline{\mathcal{L}\phi}(x)|=|\overline{\chi_j\mathcal{L} \phi}(x)
-\overline{\chi_\infty\mathcal{L  }\phi}(x)|\leq C(1-\overline{\chi_j}(x))\to 0, \quad j\to \infty
$$
uniformly for $x$ with  $|x|\leq R$. Then the required statement follows by the identity
$$\ba
\overline{\mathcal{L}\phi}(x)&=\int_{\mbR^k} \mathcal{L}\phi(x,y)\pi^{(x)}(dy)
 =\int_{\mbR^k}\left(\nabla \phi(x)\cdot a(x,y)+\frac{1}{2}\nabla^2 \phi(x)\cdot b(x,y)\right)\pi^{(x)}(dy)
\\&\hspace*{1cm}+\int_{\mbR^k}\int_{\mbR^m}\Big( \phi(x+v)-\phi(x)-\nabla \phi(x)\cdot v\Big){K}_{(\rho)}(x, \d v)\pi^{(x)}(dy)\\&\hspace*{1cm}+\int_{\mbR^k}\int_{\mbR^m}\Big( \phi(x+v)-\phi(x)\Big){K}^{(\rho)}(x,y, \d v)\pi^{(x)}(dy)
\\&=\nabla \phi(x)\cdot \overline{a}(x)+\frac{1}{2}\nabla^2 \phi(x)\cdot  \overline{b}(x)
+\int_{\mbR^m}\Big( \phi(x+v)-\phi(x)-\nabla \phi(x)\cdot v\Big)\overline{K}_{(\rho)}(x, \d v)\\&\hspace*{1cm}+\int_{\mbR^m}\Big( \phi(x+v)-\phi(x)\Big)\overline{K}^{(\rho)}(x, \d v)
\\&=L\phi(x).
\ea$$
Statement (f) can be obtained using the same `truncation of large jumps' argument as in the proof of Proposition \ref{p32} and the bounds from the assumptions $\mathbf{H}_1$, $\mathbf{H}_2$; we omit the details.

To prove statement (g), we treat $Y_\eps(t)$ as the value of the process
{$y_\eps $ from \eqref{DA_fully coupled_micro}
 taken at the (large) time instant $\tau=\eps^{-1}t$ with $\zeta_\eps(\tau) = \xi_\eps(\eps \tau)$,
 i.e.,
 $y_\eps(\tau)= Y_\eps(\eps \tau)$.} Without loss of generality we can assume that the constant $\kappa$ in the assumption  $\mathbf{H}_5$ satisfies $\kappa\leq 1$. Then by \cite[Theorem~2.8]{KulikPavlyukevich}, for every $p_Y<p+\kappa-1$,
$$
\sup_{\tau\geq 0, \eps>0}\E^\eps|y^\eps(\tau)|^{p_Y}<\infty,
$$
here we have used that the initial values $y_\eps(0)=Y_\eps(0)$ are bounded. This immediately yields (g).

\end{proof}

Now we are ready to prove \eqref{eq1prim}. Fix $N>0$, and write denote by $\P^\eps_{t-\eps N}, \E^\eps_{t-\eps N}$ the conditional probability and conditional expectation w.r.t. $\cF_{t-\eps N}^\eps$. For $\eps$ small enough, we have $s_q<t-\eps N$ and thus
$$\ba
\E^\eps \Phi( X_{\eps}(s_1), \dots,  X_\eps(s_q))&\Big(\mathcal{L}^\eps\phi(X_\eps(t),Y_\eps(t))-L\phi(X_\eps(t))\Big)
\\&=\E^\eps \Phi( X_{\eps}(s_1), \dots,  X_\eps(s_q))\E^\eps_{t-\eps N}\Big(\mathcal{L}^\eps\phi(X_\eps(t),Y_\eps(t))-L\phi(X_\eps(t))\Big).
\ea
$$
By the assumption $\mathbf{H}_0, $ there exist functions $\rho(\eps)\to 0, \varrho(\eps)\to 0$ such that
$$
\P^\eps\left(\sup_{s\in [0,T]}|\xi_\eps(s)|>\rho(\eps)\right)\leq \varrho(\eps),
$$
here $T>t$ is a fixed number. For a given $R>0$, consider the $\cF_{t-\eps N}^\eps$-measurable set
$$
\Omega_{t,N,R}^\eps=\left\{\omega: \P^\eps_{t-\eps N}\left(\sup_{s\in [0,T]}|\xi_\eps(s)|>\rho(\eps)\right)\leq R\varrho(\eps)\right\},
$$
then
by the Markov inequality 
$$\ba
\P^\eps(\Omega^\eps\setminus\Omega_{t,N,R}^\eps)&
\leq \frac{1}{ R\varrho(\eps)}\E^\eps
\left[
\P^\eps_{t-\eps N}\left(\sup_{s\in [0,T]}|\xi_\eps(s)|>\rho(\eps)\right)\right]=\frac{1}{ R\varrho(\eps)}\P^\eps \left(\sup_{s\in [0,T]}|\xi_\eps(s)|>\rho(\eps)\right)
\\&\leq \frac{\varrho(\eps)}{ R\varrho(\eps)}= \frac{1}{R}.
\ea$$
We have seen in the proof of Lemma \ref{l1} that $L\phi=\overline{\chi_\infty \mathcal{L}\phi}$,
 thus by statement (c) of this lemma the function  $L\phi$ is bounded.
The functions $\Phi, \mathcal{L}^\eps$ are bounded, as well, hence
\be\label{step1}
\ba
\Big|\E^\eps \Phi( X_{\eps}(s_1), \dots,  X_\eps(s_q))&\Big(\mathcal{L}^\eps\phi(X_\eps(t),Y_\eps(t))-L\phi(X_\eps(t))\Big)\Big|
\\&\leq C\P^\eps(|X_\eps(t-\eps N)|>R)+C\P^\eps(|Y_\eps(t-\eps N)|>R)+\frac{C}{R}
\\&+C\E^\eps 1_{\wt \Omega_{t,N,R}^\eps}\left|\E^\eps_{t-\eps N}\Big(\mathcal{L}^\eps\phi(X_\eps(t),Y_\eps(t))-L\phi(X_\eps(t))\Big)\right|,
\ea
\ee
where we denote
$$
\wt \Omega_{t,N,R}^\eps=\Omega_{t,N,R}^\eps\cap\{|X_\eps(t-\eps N)|\leq R, |Y_\eps(t-\eps N)|\leq R\}.
$$
Fix $j\geq 1$, decompose
\be\label{decomp}
\ba
\E^\eps_{t-\eps N}&\Big(\mathcal{L}^\eps\phi(X_\eps(t),Y_\eps(t))-L\phi(X_\eps(t))\Big)
\\&= \E^\eps_{t-\eps N}\Big(\mathcal{L}^\eps\phi(X_\eps(t),Y_\eps(t))-\mathcal{L}^\eps\phi(X_\eps(t),Y_\eps(t))\chi_j(X_\eps(t),Y_\eps(t))\Big)
\\&+ \E^\eps_{t-\eps N}\Big(\mathcal{L}^\eps\phi(X_\eps(t),Y_\eps(t))-\mathcal{L}\phi(X_\eps(t),Y_\eps(t))\Big)\chi_j(X_\eps(t),Y_\eps(t))
\\&+ \left(\E^\eps_{t-\eps N}\mathcal{L}\phi(X_\eps(t),Y_\eps(t))\,
\chi_j(X_\eps(t),Y_\eps(t))-P_N^{frozen}( \chi_j \mathcal{L}\phi)(X_\eps(t-\eps N),Y_\eps(t-\eps N))\right)
\\& +\Big(P_N^{frozen}( \chi_j \mathcal{L}\phi)(X_\eps(t-\eps N),Y_\eps(t-\eps N))-
\overline{\chi_j \mathcal{L}\phi }(X_\eps(t-\eps N))\Big)
\\&+ \Big(\overline{\chi_j \mathcal{L}\phi }(X_\eps(t-\eps N))-L\phi(X_\eps(t-\eps N))\Big)
\\&+ \Big(L\phi(X_\eps(t-\eps N))-L\phi(X_\eps(t))\Big).
\ea\ee

 Let us estimate each term in the decomposition \eqref{decomp}. For the first term, we simply write using Lemma \ref{l1}  (a)
\be\label{jeden}
|\E^\eps_{t-\eps N}\Big(\mathcal{L}^\eps\phi(X_\eps(t),Y_\eps(t))-\mathcal{L}^\eps\phi(X_\eps(t),Y_\eps(t))\chi_j(X_\eps(t),Y_\eps(t))\Big)|
\leq C\E^\eps_{t-\eps N}\Big(1-\chi_j(X_\eps(t),Y_\eps(t))\Big)
\ee
For the
second term, we recall that the support of $\chi_j$ is compactly embedded to $B$, thus by Lemma \ref{l1} (b)
\be\label{dwa}
\ba
\Big|\E^\eps_{t-\eps N}\Big(\mathcal{L}^\eps\phi(X_\eps(t),Y_\eps(t))&-\mathcal{L}\phi(X_\eps(t),Y_\eps(t))\Big)\chi_j(X_\eps(t),Y_\eps(t))\Big|
\\&\leq \sup_{x,y}\Big|\mathcal{L}^\eps\phi(x,y)-\mathcal{L}\phi(x,y))\Big|\chi_j(x,y)\to 0, \quad \eps\to 0.
\ea\ee
To estimate the third term in \eqref{decomp},  observe first that the function  $\chi_j \mathcal{L}\phi$ is continuous, which folows from $\mathbf{H}_4,$ $\mathbf{H}_6$ similarly to Lemma \ref{l1} (b). Next,   define the  pair
$x_\eps, y_\eps$ by \eqref{micro} with $t_0=t-\eps N$ and take
 $\wt \P^\eps=\P^\eps_{t-\eps N, \omega}$, the regular version of the conditional probability$.$ Then, for a.a. $\omega\in \wt \Omega_{t,N,R}^\eps$, the pair $x_\eps, y_\eps$ w.r.t. the probability $\P^\eps_{t-\eps N, \omega}$ belongs to the class $\mathcal{K}(\rho, 2R\varrho, R,2N)$ in the notation introduced before Proposition \ref{p32}.  Applying this Proposition, we get
\be\label{trzy}
\E^\eps1_{\wt \Omega_{t,N,R}^\eps}\left|\E^\eps_{t-\eps N}\mathcal{L}\phi(X_\eps(t),Y_\eps(t))\,
 \chi_j(X_\eps(t),Y_\eps(t))-P_N^{frozen}(\chi_j \mathcal{L}\phi )(X_\eps(t-\eps N),Y_\eps(t-\eps N))\right|\to 0,
 \quad \eps\to 0.
\ee
To estimate the fourth term, we use Proposition \ref{p1}; without loss of generality we assume that $\kappa<1$. Since the function $\chi_j \mathcal{L}\phi $ is bounded, Proposition \ref{p1} yields
\be\label{cztery}
\E^\eps1_{\wt \Omega_{t,N,R}^\eps}\Big|P_N^{frozen}(\chi_j \mathcal{L}\phi )(X_\eps(t-\eps N),Y_\eps(t-\eps N))-\overline{\chi_j \mathcal{L}\phi}(X_\eps(t-\eps N))\Big|\leq CN^{-\frac{p+\kappa-1}{1-\kappa}}.
\ee
For the fifth term, we have simply
\be\label{piac}
\E^\eps1_{\wt \Omega_{t,N,R}^\eps}\Big|(\overline{\chi_j \mathcal{L}\phi }(X_\eps(t-\eps N))-L\phi(X_\eps(t-\eps N))\Big|\leq \sup_{|x|\leq R}|\overline{ \chi_j \mathcal{L}\phi}(x)- L\phi(x)|.
\ee
For the sixth term, we have simply
\be\label{szesc}
\E^\eps 1_{\wt \Omega_{t,N,R}^\eps} \Big|L\phi(X_\eps(t-\eps N))-L\phi(X_\eps(t))\Big|\to 0, \quad \eps\to 0
\ee
by Lemma \ref{l1},(f) and uniform continuity of  $L\phi$ on compacts. Summarizing the estimates \eqref{step1} and \eqref{jeden} -- \eqref{szesc}, we get
\be\label{step2}
\ba
\limsup_{\eps\to 0}|&\E^\eps \Phi( X_{\eps}(s_1), \dots,  X_\eps(s_q))\Big(\mathcal{L}^\eps\phi(X_\eps(t),Y_\eps(t))-L\phi(X_\eps(t))\Big)|
\\&\leq C\sup_{s\leq t, \eps>0}\P^\eps(|X_\eps(s)|>R)+C\sup_{s\leq t, \eps>0}\P^\eps(|Y_\eps(s)|>R)+\frac{C}{R}
\\&+CN^{-\frac{p+\kappa-1}{1-\kappa}}+\sup_{|x|\leq R}|\overline{\mathcal{L}\phi \chi_j}(x)- L\phi(x)|
\\&+C \limsup_{\eps\to 0}\E^\eps 1_{\wt \Omega_{t,N,R}^\eps} E^\eps_{t-\eps N}\Big(1-\chi_j(X_\eps(t),Y_\eps(t))\Big)
\ea\ee
Similarly to \eqref{trzy} -- \eqref{piac}, we have
$$
\limsup_{\eps\to 0}\E^\eps 1_{\wt \Omega_{t,N,R}^\eps} E^\eps_{t-\eps N}\Big(1-\chi_j(X_\eps(t),Y_\eps(t))\Big)\leq
CN^{-\frac{p+\kappa-1}{1-\kappa}}+\sup_{|x|\leq R}(1-\overline{\chi_j}(x)),
$$
thus
\be\label{step3}
\ba
\limsup_{\eps\to 0}|&\E^\eps \Phi( X_{\eps}(s_1), \dots,  X_\eps(s_q))\Big(\mathcal{L}^\eps\phi(X_\eps(t),Y_\eps(t))-L\phi(X_\eps(t))\Big)|
\\&\leq C\sup_{s\leq t, \eps>0}\P^\eps(|X_\eps(s)|>R)+C\sup_{s\leq t, \eps>0}\P^\eps(|Y_\eps(s)|>R)+\frac{C}{R}
\\&+CN^{-\frac{p+\kappa-1}{1-\kappa}}+\sup_{|x|\leq R}|\overline{\mathcal{L}\phi \chi_j}(x)- L\phi(x)|+C\sup_{|x|\leq R}(1-\overline{\chi_j}(x)).
\ea\ee
The constants $R,N,j$ in the above inequality are arbitrary. Taking first $j\to\infty, N\to \infty$ for a fixed $R$,  we get by Lemma \ref{l1}(d),(e) that
we get
\be\label{step4}
\ba
\limsup_{\eps\to 0}|&\E^\eps \Phi( X_{\eps}(s_1), \dots,  X_\eps(s_q))\Big(\mathcal{L}^\eps\phi(X_\eps(t),Y_\eps(t))-L\phi(X_\eps(t))\Big)|
\\&\leq C\sup_{s\leq t, \eps>0}\P^\eps(|X_\eps(s)|>R)+C\sup_{s\leq t, \eps>0}\P^\eps(|Y_\eps(s)|>R)+\frac{C}{R}.
\ea\ee
Then by Lemma \ref{l1} (f),(g) we can pass to the limit $R\to \infty$ and finally get
$$
\limsup_{\eps\to 0}\left|\E^\eps \Phi( X_{\eps}(s_1), \dots,  X_\eps(s_q))\Big(\mathcal{L}^\eps\phi(X_\eps(t),Y_\eps(t))-L\phi(X_\eps(t))\Big)\right|=0.
$$
This proves \eqref{eq1prim} and completes the entire proof.

\section{Appendix. } \label{Proofs}
\begin{proof}[Proof of Lemma \ref{lem:est_exit_time}]
Set 
$\zeta_\ve(t)=\inf\{s\geq 0\ :\ \int_0^s b_\ve^2(z)dz\geq t\}.$
 Making the change of time $ \tilde \eta_{\ve }(t):=\eta_{\ve }(\zeta_\ve(t)),$ we see that
 $ \tilde \eta_{\ve }(t)$ satisfies assumptions of this Lemma with another constant $\tilde A>0$  and
a new Wiener process $\tilde W(t)=W(\zeta_\ve(t))$ but with $\tilde b_\ve(t)\equiv 1.$
Since $(C_2)^{-2}t\leq \zeta_\ve(t)\leq (C_1)^{-2}t,$ without loss of generality we will
 assume that
  $ b_\ve(t)\equiv 1.$

  Set $L_\ve:= Ax^\gamma\frac{d}{dx} +\frac{\ve^2}{2} \frac{d^2}{dx^2}. $
Denote
\[
v_\ve(x):= \int_0^{|x|}\exp\{ \frac{-2Ay^{\gamma+1}}{({\gamma+1})\ve^2} \}
\left(\int_0^y \frac{2}{\ve^2}\exp\{ \frac{2Az^{\gamma+1}}{({\gamma+1})\ve^2}\}dz  \right)dy.
\]
We have $L_\ve v_\ve(x)\geq1, $ $\sgn(x) v'_\ve(x)\geq 0,$ and $v_\ve(0)=0.$

Then by Ito's formula we have
\[
\E v_\ve(\eta_\ve(\tau_\ve(\delta)\wedge n))= \E\int_0^{\tau_\ve(\delta)\wedge n}\big(
a_\ve(s)\eta_\ve^\gamma(s)v'_\ve(\eta_\ve(s))   + \frac{\ve^2}{2}v''_\ve(\eta_\ve(s))  \big)ds\geq
\]
\[
\E\int_0^{\tau_\ve(\delta)\wedge n}\big(
A\eta_\ve^\gamma(s)v'_\ve(\eta_\ve(s))   + \frac{\ve^2}{2}v''_\ve(\eta_\ve(s))  \big)ds
= \E\int_0^{\tau_\ve(\delta)\wedge n}L_\ve v_\ve(\eta_\ve(s))   ds\geq
\]
\[
 \E\int_0^{\tau_\ve(\delta)\wedge n}1   ds=
\E {\tau_\ve(\delta)\wedge n}.
 \]
Passing $n\to\infty$ and applying the Fatou lemma we get a.s. finiteness of $\tau_\ve(\delta)$. Since $v_\ve(\eta_\ve(\tau_\ve(\delta)))=v_\ve(\delta)=v_\ve(-\delta)$, we get the  estimate
\[
\E \tau_\ve(\delta)\leq  v_\ve(\delta).
\]

Let $x>0$ be arbitrary.
 Changing the variables
$s:=\frac{z^{\gamma+1}}{\ve^2}$ and $t:=\frac{y^{\gamma+1}}{\ve^2}$
we get
\be\label{eq:1407}
\begin{aligned}
v_\ve(x)=
\frac{2\ve^{\frac{2}{\gamma+1}}}{(\gamma +1)\ve^2}
  \int_0^{\frac{|x|^{\gamma+1}}{\ve^2}}\exp\{-2At/(\gamma+1) \}
\left(\int_0^{t^{\frac{1}{\gamma+1}}\ve^{\frac{2}{\gamma+1}}}
 \exp\{ \frac{2Az^{\gamma+1}}{(\gamma+1)\ve^2}\}dz  \right)
t^{\frac{-\gamma}{\gamma+1}}dt=
\\
\frac{2\ve^{\frac{4}{\gamma+1 }}}{(\gamma +1)^2\ve^2}
\int_0^{\frac{|x|^{\gamma+1}}{\ve^2}}\exp\{-2At/(\gamma+1) \}
\left(\int_0^t  \exp\{  2As/(\gamma+1)\} s^{\frac{-\gamma}{\gamma+1}} ds  \right)
t^{\frac{-\gamma}{\gamma+1}}dt=
\\
 \frac{2 }{(\gamma +1)^2 } \ve^{\frac{2(1-\gamma)}{\gamma+1}} \int_0^{\frac{|x|^{\gamma+1}}{\ve^2}}\exp\{-2At/(\gamma+1) \}
\left(\int_0^t  \exp\{  2As/(\gamma+1)\} s^{\frac{-\gamma}{\gamma+1}} ds  \right)
t^{\frac{-\gamma}{\gamma+1}}dt.
\end{aligned}
\ee
  It follows from
L'H\^opital's rule
that
for any   $\alpha>0$ and $\beta>-1$:
\[
\int_0^te^{\alpha s} s^\beta ds \sim \alpha^{-1}e^{\alpha t} t^\beta , \ t\to+\infty.
\]
So
\[
\int_0^t  \exp\{  2As/(\gamma+1)\} s^{\frac{-\gamma}{\gamma+1}} ds \sim
\frac{\gamma+1}{2A}\, \exp\{  2At/(\gamma+1)\} t^{\frac{-\gamma}{\gamma+1}}, \ t\to+\infty.
\]
Applying this and L'H\^opital's rule
we get
\[
\ba
\int_0^{u} \exp\{-2At/(\gamma+1) \}
\left(\int_0^t  \exp\{  2As/(\gamma+1)\} s^{\frac{-\gamma}{\gamma+1}} ds  \right)
t^{\frac{-\gamma}{\gamma+1}}dt\sim
\\
\frac{\gamma+1}{2A}\, \int_0^{u}\exp\{-2At/(\gamma+1) \}
\left(  \exp\{  2At/(\gamma+1)\} t^{\frac{-\gamma}{\gamma+1}}  \right)
t^{\frac{-\gamma}{\gamma+1}}dt=\\
\frac{\gamma+1}{2A}\, \int_0^{u}
 t^{\frac{-2\gamma}{\gamma+1}}
 dt =
\frac{(\gamma+1)^2}{2A(1-\gamma)}   u^{\frac{1-\gamma}{\gamma+1}}, \ u\to+\infty.
\ea
\]
Therefore, we get from \eqref{eq:1407}   the following equivalence
  for
any fixed $x\neq 0$ as $\ve\to0:$
\[
\begin{aligned}
v_\ve(x) \underset{\ve\to0}\sim  K   \ve^{\frac{2(1-\gamma)}{\gamma+1}} \left(
\frac{|x|^{\gamma+1}}{\ve^2}\right)^{{\frac{-2\gamma}{\gamma+1}} +1}
 = K_2  \ve^{\frac{2(1-\gamma)}{\gamma+1}} \left(
\frac{|x|^{\gamma+1}}{\ve^2}\right)^{{\frac{1-\gamma}{\gamma+1}} }=
  K_2  \ve^{\frac{2(1-\gamma)}{\gamma+1}}
  {|x|^{1-\gamma }}
 \ve^{{\frac{ 2(\gamma-1)}{\gamma+1}}} = K_2
  {|x|^{1-\gamma }},
\end{aligned}
\]
where  $K$ is a  constant independent of $\delta.$

This yields that for any fixed $\delta\geq 0$:
\[
\limsup_{\ve\to0}\E \tau_\ve(\delta)\leq \limsup_{\ve\to0}\E  v_\ve(\delta)=
K_2 \delta^{1-\gamma }.
\]
This completes the proof of the Lemma.
\end{proof}

   \end{document}